\documentclass[a4paper,12pt]{amsart}
 \usepackage[T1]{fontenc}
 \usepackage[utf8]{inputenc}
\usepackage[english]{babel}
\usepackage{etoolbox}
\usepackage{booktabs}
\usepackage{subfigure}
\usepackage{amsmath, amssymb, amsthm, amsfonts}
\usepackage{mathrsfs}
\usepackage{tikz}
\usepackage{graphicx}    
\usepackage{caption}     
\usepackage{subcaption}

\usetikzlibrary{arrows}
\usetikzlibrary{matrix}
\usepackage[left=1.8cm, right=1.8cm, top=2cm]{geometry}
\usepackage{enumerate}
\usepackage{hyperref}
\usepackage{algorithm}
\usepackage{algorithmic}
\date{}
\usepackage[font=small,labelfont=bf]{caption}
\theoremstyle{plain}

\newtheorem{theorem}{Theorem}[section] 
\newtheorem{lemma}{Lemma}[section]
\newtheorem{Rule}{Rule}
\newtheorem{cor}{Corollary}[section]

\nonstopmode\numberwithin{equation}{section}
\theoremstyle{definition}
 
\newtheorem{rema}{Remark}[section]
\newtheorem{assum}{Assumption}[section]

\newtheorem*{ques*}{Question}

\makeatletter
\def\tagform@#1{\maketag@@@{\ignorespaces#1\unskip\@@italiccorr}}
\let\orgtheequation\theequation
\def\theequation{(\orgtheequation)}
\makeatother
\let\orgautoref\autoref

\renewcommand{\autoref}[1]{\def\equationautorefname{}\orgautoref{#1}}

\usepackage{fancyhdr}

\newcommand\shorttitle{Stochastic Bouligand Landweber}

\begin{document}

\title[Hanke-Raus heuristic rule for IRSGD]{Hanke-Raus heuristic rule for iteratively regularized stochastic gradient descent}
\author{  Harshit Bajpai$^{\dagger}$, Gaurav Mittal$^{\ddagger}$, Ankik Kumar Giri$^{\dagger}$}
 \email{bajpaiharshit87@gmail.com, gaurav.mittaltwins@yahoo.com, ankik.giri@ma.iitr.ac.in}
 \address{$^\dagger$Department of Mathematics, Indian Institute of Technology Roorkee, Roorkee, Uttarakhand, 247667, India}
 \address{$^\ddagger$Defence Research and Development Organization, Near Metcalfe House, Delhi,  110054, India}
\maketitle
\begin{quote}
{{\em \bf Abstract.}} Over the past decade, stochastic algorithms have emerged as scalable and efficient tools for solving large-scale ill-posed inverse problems by randomly selecting subsets of equations at each iteration. However, due to the ill-posedness and measurement noise, these methods often suffer from oscillations and semi-convergence behavior, posing challenges in achieving stable and accurate reconstructions. This study proposes a novel variant of the stochastic gradient descent (SGD) approach, the iteratively regularized stochastic gradient descent (IRSGD) to address nonlinear ill-posed problems in Hilbert spaces. Under standard assumptions, we demonstrate that the mean square iteration error of the method tends to zero for exact data.  In the presence of noisy data, we first propose a heuristic parameter choice rule (HPCR) and then apply the IRSGD method in combination with HPCR. Precisely, HPCR selects the regularization parameter without requiring any a-priori information of the noise level. We show that the method terminates in finitely many steps in case of noisy data and has regularizing features.
Finally, some numerical experiments are performed to demonstrate the practical efficacy of the method.
\end{quote}

\vspace{.3cm}

\noindent



 \section{Introduction}\noindent
This work focuses on deriving an approximate solution for a system of nonlinear ill-posed equations
\begin{equation}\label{ststart}
    F_{i}(u) = y^{\dagger}_i,
\end{equation}
where $i = 0, 1,\ldots, P-1,$ and for each $i$, the mapping  $F_i : \mathfrak{D}(F_i) \subset U \rightarrow Y_i$  represents a nonlinear operator between the Hilbert spaces $U$ and $Y_i$ with domain $\mathfrak{D}(F_i)$. In this context, $y_{i}^{\dagger} \in Y_i$ represents the exact data, $P\in \mathbb{N}$,  $U$ and $Y_i$ are endowed with their standard inner products and norms, respectively. Systems of the form \ref{ststart} commonly appear in various real-life applications. For instance, they are fundamental in various tomography methods used in fields such as medical imaging and industrial diagnostics \cite{hanafy1991quantitative,jin2023convergence,natterer2001mathematics}.
Conventionally, for the space $Y^P=Y_0\times Y_1 \times\cdots \times Y_{P-1}$  equipped with the natural inner product derived from those of $Y_i$,  \ref{ststart} can be interpreted as 
\begin{equation}\label{stcombined}
  F(u) = y^{\dagger} , 
\end{equation}
where the mapping $F : U \rightarrow Y^P$ is given by
$$ F(u) = \begin{pmatrix} 
           F_0(u) \\
           \vdots \\
           F_{P-1}(u)
          \end{pmatrix},\ \  \text{and}\ \hspace{2mm} 
         y^{\dagger}= \begin{pmatrix} 
           y^{\dagger}_0 \\
           \vdots \\
           y^{\dagger}_{P-1}
          \end{pmatrix}. $$
The exact data $y^\dagger$ is typically never available. Rather, a noisy data  $y^{\delta} \in Y^{P}$ with the noise level $\delta \geq 0$ is accessible such that 
\begin{equation}\label{y and y delta}
    \|y^{\dagger} - y^{\delta} \| \leq \delta.
\end{equation}
Here, we write $\delta = \sqrt{\delta_{0}^2 + \delta_{1}^2 + \cdots + \delta_{P-1}^2}$ and $y^{\delta}= (y^{\delta}_0, y^{\delta}_1,...,
y^{\delta}_{P-1})$, where
\begin{equation}
    \|y_{i}^{\delta} - y_i^{\dagger}\| \leq \delta_i\ \ \text{for}\ \ i = 0, 1, \ldots,P-1.
\end{equation}
 The solution to \ref{ststart} might not exist because of its ill-posedness, and even if it does, it is unlikely to be unique. Moreover, the solution(s) can be highly sensitive to perturbations in  
 $y^{\delta}$ (see \cite{engl1996regularization}). Therefore, regularization techniques are required to find the stable approximate solution of \ref{ststart}. \newline
 Iterative regularization is a well-established and highly effective class of regularization techniques. This approach has achieved significant success in solving a wide variety of inverse problems, (see \cite{engl1996regularization, kaltenbacher2008iterative}). Among these methods, the \emph{Landweber Iteration Method} (LIM) is a widely recognized classical iterative technique. For the mapping $F: U \rightarrow Y^{P}$ that is Fr\'echet differentiable, LIM is
 \begin{equation}\label{LB}
     u_{k+1}^{\delta} = u_{k}^{\delta} - F'(u_k^{\delta})^*(F(u_k^{\delta}) - y^{\delta}), \hspace{10mm} k \geq 0,
 \end{equation}
where $u_0^{\delta}:=u_{0 }$ is an initial guess and $F'(u_k^{\delta})^*$ represents the adjoint of the Fr\'echet derivative of $F(u_k^{\delta})$.
 
 The method \ref{LB} was modified by Scherzer \cite{scherzer1998modified}, who included a damping term and called it the \emph{iteratively regularized Landweber iteration method}. For  $\lambda_k\geq 0$, the iteration takes the form
      \begin{equation}\label{damping}
     u_{k+1}^{\delta} = u_{k}^{\delta} - F'(u_k^{\delta})^*(F(u_k^{\delta}) - y^{\delta}) - \lambda_{k}(u_k^{\delta} - u^{(0)}),  
 \end{equation}
where $k \geq 0$. The \(k\)-th iteration of \eqref{damping} can be viewed as performing a gradient descent step on the functional
\begin{equation*}\label{Landweber functional}
    J(u) = \frac{1}{2} \left( \|F(u) - y^\delta\|^2 + \lambda_k \|u - u^{(0)}\|^2 \right ) = \frac{1}{2} \left( \sum_{i=0}^{P-1}\|F_{i}(u) - y_{i}^\delta\|^2 + \lambda_k \|u - u^{(0)}\|^2 \right ).
\end{equation*}
Compared to the method \ref{LB}, the analysis of deriving convergence rates for the method \ref{damping} requires less assumptions on $F$. Moreover, the method \ref{damping} converges to a solution that remains close to $u^{(0)}$, whereas achieving a similar convergence behavior with the method \ref{LB} needs additional assumptions (see \cite{kaltenbacher2008iterative}). This highlights the advantage of incorporating the extra damping term in \ref{damping} as compared to \ref{LB}. 
 When the noisy data is available, it is necessary to use an appropriate stopping rule (a-posteriori) to illustrate that the iterative schemes \ref{LB} and \ref{damping} are regularization methods. The prevalent stopping rule is referred to as \textit{discrepancy principle}~\cite{Morozov1966}, which states that the method is stopped after $k_\delta = k_\delta (\delta, y^\delta)$ steps, where
\begin{equation}\label{discrepancy}
     \|F(u_{k_\delta}^\delta) - y^{\delta}\| \leq \tau \delta <  \|F(u_{k}^{\delta}) - y^{\delta}\|, \hspace{10mm} 0 \leq k < k_\delta, 
\end{equation}
 for some $\tau >1.$
  However, it is challenging to apply the discrepancy principle \ref{discrepancy} in practicality since the noise level is not always reliable or available. Consequently, while applying the discrepancy principle, overestimation or underestimation of the noise level could lead to serious deterioration of the reconstruction accuracy. Therefore, it is reasonable to work on heuristic stopping rules that do not depend on the information of noise level. Real \cite{real2024hanke} proposed a heuristic rule for the Landweber iteration \ref{LB}  based on the discrepancy principle, which is further motivated by the method of Hanke and Raus \cite{hanke1996generalheuristic}. For \ref{LB}, heuristic rule determines an integer $k_* := k_{*}(y^{\delta})$ by minimizing
\begin{equation}\label{heuristic previous}
    \Psi (k, y^\delta) := (k+a)\|F({u_k^\delta}) - y^\delta\|^{p}
\end{equation}
over $k$, where  $p>0$ is a constant and $a \geq 1$ is a fixed number.

It may be observed that the significant computational cost per iteration is a common problem with traditional iterative regularization techniques mainly due to the need to process complete dataset. Applying stochastic gradient descent (SGD) \cite{jin2018regularizing, jin2020convergence, jin2021saturation, robbins1951stochastic} is a viable and successful way to deal with this problem. Using \cite{jin2020convergence}, the SGD method can be expressed as
\begin{equation}\label{SGD}
    u_{k+1}^{\delta} = u_{k}^{\delta} - \mu_{k}F_{i_k}'(u_k^{\delta})^*(F_{i_k}(u_k^{\delta}) - y_{i_k}^{\delta}), \hspace{10mm} k \geq 0,
 \end{equation}
where $i_{k}$ is a uniformly drawn random index from $\{0,1, \ldots, P-1\}$ and  $\mu_{k} \geq 0$ is the step size. For the SGD method \ref{SGD}, the stopping index $k(\delta)\in \mathbb{N}$ satisfies the rule
\begin{equation}\label{aprioriforSGD} \lim_{\delta \rightarrow 0^{+}} k(\delta) = \infty \hspace{5mm} \text{and} \hspace{5mm} \lim_{\delta \rightarrow 0^{+}}\delta^2 \sum_{i=1}^{k(\delta)} \mu_{i} = 0.
\end{equation}
Notably, the aforementioned rule is an \textit{a-priori} rule.
We refer to \cite{jin2020convergence,lu2022stochastic} for some recent work on the SGD method using \textit{a-priori} stopping rule and \cite{jahn2020discrepancy, zhang2023discrepancy} for the SGD method using \textit{a-posteriori} stopping rule. To the best of our knowledge, convergence results of SGD method via heuristic stopping rules are not yet available in the literature. Accordingly, the primary objective of this study is to explore the convergence analysis of the following iteratively regularized stochastic gradient descent (IRSGD) under heuristic stopping rules:
 \begin{equation}\label{SIRLI 2}
      u_{k+1}^{\delta} = u_{k}^{\delta} - \omega_{k}F_{i_k}'(u_{k}^{\delta})^*(F_{i_k}(u_k^{\delta}) - y_{i_k}^{\delta}) - \lambda_{k}(u_k^{\delta} - u_0), \hspace{5mm} k \geq 0,
 \end{equation}
\[\hspace{30mm} \text{with}\quad 0 \leq \lambda_k \leq \lambda_{\text{max}} <\frac{1}{2},\]
 where  $\lambda_{k}$ is the weighted parameter, $\omega_{k}$ is the step size. This method differs from the SGD method in the sense that an extra damping term $-\lambda_{k}(u_k^{\delta} - u_0)$ is 
 considered in its formulation. 
A key advantage of studying IRSGD lies in its ability to reflect SGD by setting the damping factor $\lambda_k = 0$. Our approach is novel in the realm of stochastic methods for solving ill-posed problems as it provides convergence analysis without relying on an \textit{a-priori} stopping rule. Additionally, it circumvents the need to know the noise level, which further enhances its practical applicability.

We perform the detailed convergence analysis of the proposed method with early stopping, which is guaranteed to terminate after a finite number of iterations. Unlike the a priori stopping rules, which require a deterministic stopping index, our method yields a stopping index that is inherently random due to its dependence on the sample paths. This randomness introduces analytical challenges that cannot be addressed using the techniques in \cite{jin2020convergence,jin2023convergence,lu2022stochastic}. To overcome these, we develop a convergence framework grounded in Hilbert space theory and probability.

Our convergence framework, which addresses the random stopping index defined by a heuristic rule, is inspired by the recent work \cite{Huang2025early}, where random stopping based on the discrepancy principle is studied in the context of stochastic mirror descent. By constructing a probability-one event, pathwise convergence is established for the exact data case, and, together with a stability result, this leads to almost sure convergence in the noisy case as the noise level approaches zero.
\vspace{4mm}

The main outcomes of this study are listed as follows:
     
   
    
\begin{itemize}
\item The proposed method is fully data-driven, requiring neither \textit{a-priori} nor \textit{a-posteriori} stopping rules. To the best of our knowledge, this is the first work to formulate and analyze a stochastic method with a heuristic-type stopping criterion.  

\item Unlike prior works \cite{jahn2020discrepancy, jin2018regularizing, jin2020convergence, jin2021saturation} that employ a polynomially decaying step size, we allow the step size $\omega_k$ to vary arbitrarily within a fixed interval $[\omega, \Omega]$. We establish convergence of the proposed IRSGD scheme \eqref{SIRLI 2} in Hilbert spaces under a heuristic stopping rule. The analysis is technically challenging due to the randomness of the stopping index and the inclusion of a damping term. 

\item We extend the scope from analyzing a single equation to developing and studying stochastic methods for solving large-scale systems of equations (i.e., with large $P$ in \eqref{ststart}). By randomly selecting and processing individual equations from \eqref{ststart} independently, our approach drastically reduces memory requirements, making it particularly relevant to large-scale inverse problems arising in practical applications \cite{hanafy1991quantitative, natterer2001mathematics, olafsson2006radon}.

\item Through extensive numerical experiments, we demonstrate that incorporating a damping factor into the SGD framework along with the heuristic stopping significantly  improves both reconstruction quality and convergence speed compared to standard SGD variants, confirming the practical benefits of the proposed approach.  
\end{itemize}

\textbf{Outline.}
Section~\ref{sec-3} introduces the required assumptions, presents the IRSGD algorithm, and develops the analysis for the heuristic stopping rule, including a result ensuring the iterates remain within a certain ball. Section~\ref{sec-convo} provides convergence results for both noisy and exact data. Section~\ref{Section:NE} reports numerical experiments illustrating the effectiveness of the method, and the final section concludes the paper.

\section{ Iteratively regularized Stochastic gradient descent  in Hilbert spaces}\label{sec-3}
To investigate the convergence behavior of the iterates \( u_k^\delta \) in the IRSGD method, we consider the mean squared norm \( \mathbb{E}[\|\cdot\|^2] \). In this context,  the expectation \( \mathbb{E}[\cdot] \) is taken with respect to the filtration \( \mathcal{F}_k \), which is determined by the sequence of random indices \( \{i_1, \dots, i_{k - 1}\} \) (see \cite{jin2023convergence}). 
We denote by
 $B_{\rho}(u)$ the closed ball in $U$ of radius $\rho$ centered at $u$, where $u \in U$ and $\rho > 0$. Furthermore, let $\mathcal{D}_{\rho}(u^{\dagger})$ represent the set of all solutions to \ref{stcombined} within the closed ball $B_\rho (u^{\dagger})$, which is formally given by 
 $$ \mathcal{D}_{\rho}(u^{\dagger}):= \{u \in B_{\rho}(u^{\dagger}): F(u) = y^{\dagger}\}.$$

The remainder of this section is structured as follows: Subsection~\ref{subsec-3.1} introduces the foundational assumptions required for the convergence analysis and provides a detailed exposition of the IRSGD method, including a heuristic stopping criterion as presented in Algorithm~\ref{alg:IRSGD Algorithm}. Subsection~\ref{subsec-3.2} discusses a heuristic stopping rule based on a modified discrepancy principle. In this context, we establish the almost sure finiteness of the stopping index defined by this principle and present several key theoretical results related to the proposed scheme. 

 \subsection{Assumptions}\label{subsec-3.1}
Our analysis demands the following assumptions.
\begin{assum}\label{common assumptions}
  The following is true for $i \in \{0,1, \ldots, P-1\}$, and $\rho>0$:
    \begin{itemize}
        \item [(i)] $F: U \rightarrow Y^P$ is a continuous operator.
        \item [(ii)] There exists a constant $0 < L \leq \frac{1}{2\omega}$ such that  
        \begin{equation} \label{bddness of F_i'}
            \|F_{i}'(u)\| \leq L \ \text{for all}\ u \in B_{\rho+c(\rho)}(u_0).
        \end{equation}  (see Lemma \ref{first lemma} for the meaning of $c(\rho)$ and $\omega$.) Moreover,  constant $\eta$ exists  such that 
        \begin{equation}\label{tangential}
        \|F_{i}(\overline{u}) - F_{i}(u) - F_{i}'(u)(\overline{u} - u) \| \leq \eta \|F_{i}(\overline{u}) - F_{i}(u)\|, \ \ 0<\eta<1
        \end{equation}
        for all $\overline{u}, u \in B_{\rho+c(\rho)}(u_0).$ 
        \item [(iii)] \ref{stcombined} has a solution $\hat{u} \in B_{\rho}(u_0).$
    \end{itemize}
    \end{assum}
The conditions in Assumption~\ref{common assumptions}, including the tangential cone condition~\eqref{tangential} that quantifies the nonlinearity of the forward operator 
\(F_i\), are standard in convergence analyses of iterative regularization for nonlinear inverse problems~\cite{hanke1995convergence,jin2020convergence,kaltenbacher2008iterative} and have been verified for various settings such as nonlinear integral equations and parameter identification in PDEs~\cite{hanke1995convergence}.

Motivated by \cite{hanke1996generalheuristic}, we consider the following heuristic stopping rule for the IRSGD method.

\begin{Rule}[Heuristic]\label{Hanke Rule}
  
Let $a \geq 1$ be a fixed constant such that 
\[
\Psi(k, y^{\delta}) := (k + a) \sum_{i=0}^{P-1} \|F_i(u^\delta_k) - y_i^\delta\|^2.
\] We define $k_* := k_*(y^\delta)$ as the integer satisfying
\[
k_* \in \arg\min \{ \Psi(k, y^\delta) : 0 \leq k \leq  k_\infty\},
\]
where $k_\infty := k_\infty(y^\delta)$ denotes the largest integer for which $u^\delta_k \in \mathfrak{D}(F)$ holds for all $0 \leq k  <  k_\infty$.
\end{Rule}
Using the integer $k_* := k_*(y^\delta)$ defined by Rule \ref{Hanke Rule}, we call $u^\delta_{k_*(y^\delta)}$ as an approximate solution. This gives rise to the main algorithm of the paper, presented as follows.
\begin{algorithm}
\caption{IRSGD with heuristic stopping}
\label{alg:IRSGD Algorithm}
\begin{algorithmic}[1]
\STATE{\textbf{Given:} $F_i, y_i^\delta$ for all $i \in \{0, 1, \ldots, P-1\}$ and a constant $a\geq 1$}
\STATE{\textbf{Initialize:} $u_0^\delta = u_0$, set $k = 0$}
\REPEAT
    \STATE{Choose $\omega_k \in [\omega, \Omega]$ and the weight parameter $\lambda_k$ such that
    \[
    0 \leq \lambda_k \leq \lambda_{\text{max}} < \frac{1}{2}.
    \]}
    \STATE{Randomly pick an index $i_k$ from $\{0, 1, \ldots, P-1\}$ uniformly.}
    \STATE{Update
    \[
    u_{k+1}^\delta = u_k^\delta - \omega_k F'_{i_k}(u_k^\delta)^* (F_{i_k}(u_k^\delta) - y_{i_k}^\delta) - \lambda_k (u_k^\delta - u_0),
    \]
    where $F'_{i_k}$ is the Fréchet derivative of $F_{i_k}$, and $u_0$ is the initial guess.}
    \STATE{Set $k \gets k + 1$}
\UNTIL{\[
\Psi(k, y^{\delta}) := (k + a) \sum_{i=0}^{P-1} \|F_i(u^\delta_k) - y_i^\delta\|^2
\] does not achieve its minimum (for the details see heuristic rule \ref{Hanke Rule})}
\RETURN{$u^\delta_{k_*(y^\delta)}$}
\end{algorithmic}
\end{algorithm}

\begin{rema}
Rule \ref{Hanke Rule} is straightforward to implement in practice, as detailed in Section \ref{Section:NE}. Throughout the iterative procedure, the value of $\Psi(k, y^\delta)$ is monitored as a function of $k$. After a sufficient number of iterations, the process is terminated and $k_*$ is selected that minimizes $\Psi(k, y^\delta)$. However, an important challenge associated with the nonlinear Landweber iteration \cite{HumberS} lies in determining an appropriate upper bound on the number of iterations. This is particularly critical due to its local convergence properties, which can lead to premature stopping or inaccurate identification of $k_*$. Therefore, the limit on the upper bound must be sufficiently large to avoid inaccurately estimating $k_*$ at the first local minimum, which is often not the global minimum \cite{HumberS}. 
\end{rema}

\begin{rema}
Evaluating the residual 
\(\|F(u_k^\delta) - y^\delta\|^2 = \sum_{i=0}^{P-1} \|F_i(u_k^\delta) - y_i^\delta\|^2 \) 
at every iteration of Algorithm~\ref{alg:IRSGD Algorithm} for heuristic rule~\ref{Hanke Rule} can be computationally expensive. 
A practical alternative is to perform this computation at prescribed intervals, following the approach used in stochastic gradient methods~\cite{jahn2020discrepancy,johnson2013accelerating}, 
or to exploit randomized SVD techniques in problems with intrinsic low-rank structure~\cite{Kluth2019Enhanced}. 
Similar computational overheads have been reported for the discrepancy principle~\cite{Gu2025SGDpenality,Huang2025early}. Mitigating these costs in the heuristic setting, potentially by adapting the strategy in~\cite{Huang2025early}, is a promising direction for future research.
\end{rema}

Algorithm~\ref{alg:IRSGD Algorithm} leads to the question of how to ensure that $u^\delta_{k_*(y^\delta)}$ converges to a solution of \ref{stcombined} as $\delta \to 0.$  Specifically, we seek to determine the conditions under which this convergence holds for a sequence of noisy data $\{y^\delta\}$, where $y^\delta \to y^\dag$ as $\delta \to 0$. 
 
 In general, the answer is inadequate, as stated by Bakushinskii's veto \cite{Bakushinskii veto}, which asserts that heuristic rules cannot guarantee convergence in the worst-case scenario for any regularization method. Subsequently, in \cite{hanke1996generalheuristic}, a noise condition was introduced for linear problems in Hilbert spaces. This condition was later generalized  under the Hanke-Raus rule (see Assumption \ref{new assumption}) to examine the convergence of nonlinear regularization methods as well as iterative regularization methods in Banach spaces \cite{JinQ2016, JinQ2017,  real2024hanke, Zhang and Jin Q}. We also utilize this condition in our work.

\begin{assum}[Noise condition]\label{new assumption}
    The family of noisy data $\{y^\delta\}$ satisfies the condition $ \|y^\delta - y^\dagger\| \rightarrow 0$ as $\delta \rightarrow 0$. Moreover, there exists a constant $\varkappa > 0$ such that for every $y^\delta$ and for each $u \in S(y^\delta)$, where $S(y^\delta) := \{ u^\delta_k : 0 \leq k \leq k_\infty \}$ is constructed using \ref{SIRLI 2}, there holds
\[
\|y^\delta - F(u)\| \geq \varkappa \|y^\delta - y^\dagger\|.
\]
\end{assum}
It directly follows from Assumption \ref{new assumption} that Rule \ref{Hanke Rule} yields a finite integer $k_*$. Indeed, if $k_\infty$ is finite, then $k_*$ is trivially finite. Therefore, we assume $k_\infty = \infty$.  From Assumption \ref{new assumption} there holds
\[
\Psi(k, y^\delta) = (k + a) \|F(u^\delta_k) - y^\delta\|^2 \geq (k + a) \varkappa \|y^\dagger - y^\delta\|^2 \rightarrow \infty
\]
as $k \rightarrow \infty$. Consequently,  a finite integer \( k_* \) exists that achieves the minimum of \( \Psi(k, y^\delta) \).

\begin{rema}
It is worth emphasizing that Assumption~\ref{new assumption}, while abstract and difficult to verify analytically, currently offers the most practical framework for theoretical analysis, as also noted in~\cite{real2024hanke}. Its validity can, at present, only be supported through numerical experiments. A detailed numerical investigation of heuristic parameter choice rules, beyond the Hanke–Raus rule, for the nonlinear Landweber iteration is provided in~\cite{HumberS}.
\end{rema}
\subsection{ Analysis for Heuristic Rule \ref{Hanke Rule}}\label{subsec-3.2}

This subsection aims to demonstrate that
\begin{equation}\label{conofhanke}
   \mathbb{E}\left[ \Psi(k_*(y^\delta), y^\delta) \right] \rightarrow 0 \quad \text{as} \quad \delta \rightarrow 0. 
\end{equation}
We show the same in Lemma \ref{secound lemma}.
Additionally, we demonstrate that IRSGD iterates $u_k^\delta$ remains in a ball around $u_0$  under certain conditions. 

To establish \eqref{conofhanke}, we introduce an auxiliary index \(\hat{k}_\delta\), determined by the following stopping criterion.

\begin{Rule}[Modified discrepancy]\label{modified discrepancy}
The stopping index \(\hat{k}_\delta\) is defined as the smallest integer satisfying
\begin{equation}\label{newdiscrepancy}
    \| F_i(u^\delta_{\hat{k}_\delta}) - y_i^\delta \| 
    + \frac{M}{\hat{k}_\delta + a} 
    \leq \tau \| y_i^\delta - y_i^\dagger \| 
    < \| F_i(u^\delta_k) - y_i^\delta \| 
    + \frac{M}{k + a},
    \quad \forall i = 0, 1, \ldots, P-1,
\end{equation}
for all \( 0 \leq k < \hat{k}_\delta \), where
\[
    M = \frac{\varrho \, c(\rho)^2 \, \tau}{\Omega \, (1+\varrho) \, (1+\eta)},
\]
with \(\varrho > 0\), \(\Omega > 0\), \(\eta \), \(\tau > 1\), and \(c(\rho)\) defined in \eqref{tangential}, \eqref{discrepancy}, and \eqref{c rho)}, respectively.
\end{Rule}

This modified discrepancy principle represents a slight augmentation of the classical version given in \eqref{discrepancy}, achieved by incorporating an additional penalization term. This term plays a pivotal role in the convergence analysis, particularly in establishing the result in \eqref{conofhanke}, as it guarantees that \(\hat{k}_\delta \to \infty\) as \(\delta \to 0\).

\begin{rema}
It is important to emphasize that both the stopping indices \(k_*\) and \(\hat{k}_\delta\) depend on the iterates \(u_k^\delta\), which are themselves random variables. Consequently, \(k_*\) and \(\hat{k}_\delta\) are also random quantities, adding a substantial layer of complexity to the theoretical analysis compared to the deterministic settings considered in~\cite{JinQ2016, JinQ2017, real2024hanke}.
\end{rema}
We now present an iterative scheme as Algorithm~\ref{alg:IRSGD_mod_disc} based on the modified discrepancy principle introduced in Rule~\ref{modified discrepancy}.  
In the subsequent result, we prove that all iterates generated by Algorithm~\ref{alg:IRSGD_mod_disc} remain within the ball  
\(B_{c(\rho)}(\hat{u})\), where \(c(\rho)\) is defined in~\eqref{c rho)}.

 \begin{algorithm}
\caption{IRSGD with modified discrepancy stopping}
\label{alg:IRSGD_mod_disc}
\begin{algorithmic}[1]
\STATE \textbf{Given:} operators $F_i$ and noisy data $y_i^\delta$ for $i=0, 1, \dots,P-1$; constant $a\geq 1$; parameters $\omega,\Omega>0$, $\lambda_{\max}\in[0,\tfrac12)$; discrepancy parameters $\tau>1$ and $M>0$ (as in \eqref{newdiscrepancy}).
\STATE \textbf{Initialize:} $u_0^\delta = u_0$ (initial guess), set $k \leftarrow 0$.
\REPEAT
    \STATE Choose $\omega_k\in[\omega,\Omega]$ and $\lambda_k$ with $0\le \lambda_k\le \lambda_{\max}$.
    \STATE Select index $i_k$ uniformly at random from $\{0,1,\dots,P-1\}$.
    \STATE Compute update
    \[
      u_{k+1}^\delta \;=\; u_k^\delta
      - \omega_k\, F'_{i_k}(u_k^\delta)^* \big( F_{i_k}(u_k^\delta) - y_{i_k}^\delta \big)
      - \lambda_k\,(u_k^\delta - u_0).
    \]
    \STATE Set $k \leftarrow k+1$.
\UNTIL{the modified-discrepancy condition \eqref{newdiscrepancy} is satisfied for the current iterate, i.e., the index
\(\hat{k}_\delta\) defined by
\[
\| F_i(u^\delta_{\hat{k}_\delta}) - y_i^\delta \| + \frac{M}{\hat{k}_\delta + a}
\leq \tau \| y_i^\delta - y_i^\dagger \|
\quad\text{for all } i=0, 1, \dots,P-1,
\]
holds (see Rule~\ref{modified discrepancy}).}
\RETURN{$u^\delta_{\hat{k}_\delta}$}
\end{algorithmic}
\end{algorithm}

\begin{lemma}\label{first lemma}
    Let Assumption $\ref{common assumptions}$ hold and let $\Omega, \omega$ be positive constants such that $\Omega\geq \omega_k \geq \omega$. Moreover, let $\kappa \in (0, 1)$ and $\varrho>0$ be a sufficiently small constant (see Remark $\ref{rema 2}$) and 
    \begin{equation}\label{c rho)}
        c(\rho) : = \rho\lambda_{\textbf{max}} \frac{(1-\lambda_{\textbf{max}}) +\sqrt{(1- \lambda_{\textbf{max}})^2 +  (\frac{1}{1+ \varrho} - (1-\lambda_{\text{max}})^2)(1 + \kappa^{-2}\Omega^2 L^2)}}{(\frac{1}{1+ \varrho} - (1-\lambda_{\text{max}})^2)}. 
    \end{equation} 
    Then, under the following conditions 
    \begin{equation}\label{condition for first lemma}
       D:=  2\omega(1-\lambda_{\text{max}})(1 - \eta) - \Omega^2 L^2 - \kappa^2 - \tau^{-1}\Omega (2+M)(1 + \eta)   > 0,
    \end{equation} 
    \begin{equation}\label{cond 2 for first lemma}
        \sum_{k=0}^{\infty} \lambda_k \leq \lambda < \infty,
    \end{equation}
    there holds, for any $\hat{u}$ fulfilling $\ref{stcombined}$ and for any $0 \leq k < \hat{k}_{\delta}$, if $u_{k}^{\delta} \in B_{c(\rho)}(\hat{u})$, then $ u_{k+1}^{\delta} \in B_{c(\rho)}(\hat{u}) \subset B_{\rho +c(\rho)}(u_0).$
\end{lemma}
\begin{proof}
The result is proved by mathematical induction.
For the case \(k = 0\), the claim follows directly from condition~(iii) of Assumption~\ref{common assumptions}.  
Now, assume as the induction hypothesis that \(u_k^\delta \in B_{c(\rho)}(\hat{u})\) holds for all \(0 \leq k < \hat{k}_\delta\).  
We will prove that this property also holds for the subsequent iterate, i.e., \(u_{k+1}^\delta \in B_{c(\rho)}(\hat{u})\).  
From~\eqref{SIRLI 2}, we have

\begin{align}\label{3.9}
\nonumber
\|u^\delta_{k+1} - \hat{u}\|^2
 & = (1 - \lambda_k)^2 \|u^\delta_k - \hat{u}\|^2 + \lambda_k^2 \|u_0 - \hat{u}\|^2  + 2 \lambda_k (1 - \lambda_k) \langle u^\delta_k - \hat{u}, u_0 - \hat{u} \rangle \\ 
& \hspace{5mm} + \omega_k^2\|F'_{i_k}(u^\delta_k)^*(y_{i_k}^\delta - F_{i_k}(u^\delta_k))\|^2 + 2\omega_k (1 - \lambda_k) \langle y_{i_k}^\delta - F_{i_k}(u^\delta_k), F'_{i_k}(u^\delta_k)(u^\delta_k - \hat{u}) \rangle \\ \nonumber
& \hspace{5mm}+ 2\omega_k \lambda_k \langle u_0 -\hat{u}, F'_{i_k}(u^\delta_k)^*(y_{i_k}^\delta - F_{i_k}(u^\delta_k)) \rangle.
\end{align}
To estimate \ref{3.9}, we can write the last two terms of the sum as
\begin{align}\label{3.10}
    \langle y_{i_k}^\delta - F_{i_k}(u^\delta_k), F_{i_k}'(u^\delta_k)(u^\delta_k - \hat{u}) \rangle \nonumber &= \langle y_{i_k}^\delta - F_{i_k}(u^\delta_k), y_{i_k}^\delta - y_{i_k}^{\dagger} \rangle - \| y_{i_k}^\delta - F_{i_k}(u^\delta_k) \|^2 \\
& \hspace{2mm} - \langle y_{i_k}^\delta - F_{i_k}(u^\delta_k), F_{i_k}(u^\delta_k) - F_{i_k}(\hat{u}) - F_{i_k}'(u^\delta_k)(u^\delta_k - \hat{u}) \rangle,
\end{align}
\begin{align}\label{3.11}
    \nonumber
    2\omega_k \lambda_k \langle u_0 -\hat{u}, F'_{i_k}(u^\delta_k)^*(y_{i_k}^\delta - F_{i_k}(u^\delta_k)) \rangle & \leq 2\lambda_k\omega_k L \|u_0 - \hat{u}\| \|y_{i_k}^\delta - F_{i_k}(u^\delta_k)\| \\ 
    & \leq  \kappa^{-2} \lambda_k^2 \omega_k^2 L^2 \|u_0 - \hat{u}\|^2 + \kappa^2 \|y_{i_k}^\delta - F_{i_k}(u_{k}^{\delta})\|^2,
\end{align}
where we have used \ref{bddness of F_i'} and the inequality $2ab\leq a^2+b^2$ for $a, b\in \mathbb{R}$ in obtaining \ref{3.11}.

By incorporating  \ref{bddness of F_i'}, \ref{tangential}, \ref{3.10}, and \ref{3.11}  in  \ref{3.9}, we attain that
\begin{align} \label{solved in Lemma 1}
\|u^\delta_{k+1} - \hat{u}\|^2 \nonumber &\leq  (1 - \lambda_k)^2 \|u^\delta_k - \hat{u}\|^2 + \lambda_k^2 (1 + \kappa^{-2}\omega_k^2 L^2) \|u_0 - \hat{u}\|^2 \\ \nonumber
& \quad + 2 \lambda_k (1 - \lambda_k) \|u^\delta_k - \hat{u}\| \|u_0 - \hat{u}\| \\
&\quad + \|y_{i_k}^\delta - F_{i_k}(u^\delta_k)\| \big( 2\omega_k (1 - \lambda_k)(1 + \eta) \|y_{i_k}^\dagger - y_{i_k}^{\delta}\| \\ \nonumber
 & \quad- \|y_{i_k}^\delta - F_{i_k}(u^\delta_k)\| (2\omega_k(1-\lambda_k)(1 - \eta) - \omega_k^2 L^2 - \kappa^2) \big). \nonumber
\end{align}
Next, according to the definition of $\hat{k}_{\delta}$ in \ref{newdiscrepancy}, for $k < \hat{k}_{\delta}$ we further derive that
\begin{equation} \label{disc used in lemma 1}
   \|F_{i_k}(u_{k}^{\delta}) - y_{i_k}^{\delta}\|\|y_{i_k}^\dagger - y_{i_k}^{\delta}\| \leq \frac{1}{\tau} \left(\|F_{i_k}(u_{k}^{\delta}) - y_{i_k}^{\delta}\|^2 + \frac{M}{k+a} \|F_{i_k}(u_{k}^{\delta}) - y_{i_k}^{\delta}\| \right).      
\end{equation}
Using the fact that $ \omega_k \in [\omega, \Omega], u_k^\delta \in  B_{c(\rho)}(\hat{u})$ and $\hat{u} \in  B_{\rho}(u_0)$ along with inequality \ref{disc used in lemma 1} in \ref{solved in Lemma 1}, we have
\begin{align}\label{3.14}
    \|u^\delta_{k+1} - \hat{u}\|^2 \nonumber &\leq  (1 - \lambda_k)^2 c(\rho)^2 + \lambda_k^2 (1 + \kappa^{-2}\Omega^2 L^2) \rho^2 + 2 \lambda_k (1 - \lambda_k)\rho c(\rho) \\
& \quad -  \left( (2\omega_k(1-\lambda_k)(1 - \eta) - \omega_k^2 L^2 - \kappa^2) - 2\omega_k\tau^{-1} (1 - \lambda_k)(1 + \eta)  \right)\|y_{i_k}^\delta - F_{i_k}(u^\delta_k)\|^2\\ \nonumber
& \quad +\frac{2M\omega_k(1-\lambda_k)(1+\eta)\|y_{i_k}^{\delta} - F_{i_k}(u_k^\delta)\|}{\tau(k+a)}. \nonumber
\end{align}
We note that \[(1+\varrho)\big((1 - \lambda_k)^2 c^2 + \lambda_k^2 (1 + \kappa^{-2}\Omega^2 L^2) \rho^2 + 2 \lambda_k (1 - \lambda_k)\rho c\big) \leq c^2,\] 
provided
\[ c \bigg(\frac{1}{1+ \varrho} - (1-\lambda_k)^2\bigg) \geq  \rho\lambda_k (1-\lambda_k) +\sqrt{(1 -\lambda_k) ^2 +   \bigg(\frac{1}{1+ \varrho} - (1-\lambda_k)^2\bigg)(1 + \kappa^{-2}\Omega^2 L^2)},\]
which is true for $c(\rho)$. From \ref{c rho)} along with the inequality $2ab \leq a^2 + b^2, \forall \hspace{2mm} a, b \in \mathbb{R}$, \ref{3.14},  $\omega_k \in [\omega, \Omega], 0 \leq \lambda_k \leq \lambda_{\text{max}} <1$ and $1-\lambda_k \leq 1$, we get 
\begin{align*}
\|u^\delta_{k+1} - \hat{u}\|^2 \nonumber  &\leq -  \left(  (2\omega_k(1-\lambda_k)(1 - \eta) - \omega_k^2 L^2 - \kappa^2) - (2+M)\omega_k\tau^{-1} (1 - \lambda_k)(1 + \eta)  \right)\|y_{i_k}^\delta - F_{i_k}(u^\delta_k)\|^2\\
& \quad + \frac{c(\rho)^2}{1+\varrho} +\frac{M\omega_k(1-\lambda_k)(1+\eta)}{\tau(k+a)^2}\\
   &\leq  -  \left( (2\omega(1-\lambda_{\text{max}})(1 - \eta) - \Omega^2 L^2 - \kappa^2) - (2+M)\Omega\tau^{-1} (1 + \eta)  \right)\|y_{i_k}^\delta - F_{i_k}(u^\delta_k)\|^2\\
& \hspace{10mm} +\frac{c(\rho)^2}{1+\varrho} +\frac{M\Omega(1+\eta)}{\tau(k+a)^2}.
\end{align*}
This with \ref{condition for first lemma}, $M = \frac{\varrho c(\rho)^2\tau}{\Omega(1+ \varrho)(1+\eta)}$ and the inequality $k+a \geq 1$ implies that
\[ \|u^\delta_{k+1} - \hat{u}\|^2 \leq \frac{c(\rho)^2}{1+ \varrho} + \frac{\varrho c(\rho)^2}{1+ \varrho} = c(\rho)^2.\]
This means that $u_{k+1}^\delta \in B_{c(\rho)}(\hat{u})$. Further, by engaging the triangular inequality and the fact that $\|u_0 - \hat{u}\| \leq \rho$, we conclude that 
\begin{align*}
    \|u^\delta_{k+1} - u_{0}\| &\leq \|u^\delta_{k+1} - \hat{u}\| + \|\hat{u} - u_{0}\| \leq \rho + c(\rho),
\end{align*}
which implies that $u_{k+1}^\delta \in B_{\rho + c(\rho)}(u_0)$. The desired result is thus established.
\end{proof}

\begin{rema}\label{rema 2} For the definition of \( c(\rho) \) to be meaningful, there must exist a positive real root of the quadratic equation
\[
\left((1 - \lambda_k)^2 - \frac{1}{1 + \varrho}\right)c^2 + \lambda_k^2 (1 + \kappa^{-2} \Omega^2 L^2) \rho^2 + 2 \lambda_k (1 - \lambda_k) \rho c = 0,
\]
which is satisfied if the condition \( \left((1 - \lambda_k)^2 - \frac{1}{1 + \varrho}\right) < 0 \) holds. Therefore, a sufficiently small positive constant \( \varrho \) is required in \ref{c rho)}, due to the restriction that \( \lambda_k \in [0, \frac{1}{2}) \).
\end{rema}
We first show that Algorithm~\ref{alg:IRSGD_mod_disc} is well-defined. Specifically, for 
\[
\|y_i^\delta - y_i^\dagger\|_{i \in \{0, 1, \ldots, P-1\}} > 0,
\] 
the modified discrepancy principle~\eqref{newdiscrepancy} yields a finite stopping index, ensuring that the algorithm terminates in finitely many steps. Moreover, we prove that \(\hat{k}_\delta \to \infty\) as \(\delta \to 0\). Here, the expectation \(\mathbb{E}\)  and probability \(\mathbb{P}\) are taken with respect to the filtration \(\{\mathcal{F}_k\}_{k \ge 0}\), where \(\mathcal{F}_k\) is the \(\sigma\)-algebra generated by the random indices \(i_n\) for \(0 \le n < k\). The formal result is stated below.

\begin{theorem}\label{kdelta hat finite}
    Under the assumptions of Lemma~$\ref{first lemma}$, the stopping index $\hat{k}_{\delta}$ as defined in  $\ref{newdiscrepancy}$ almost surely remains finite for  $\|y_i^\delta - y_i^\dagger\|_{i \in \{0, 1, \ldots, P-1\}} > 0$ and $\hat{k}_{\delta} \rightarrow \infty$ as $\delta \rightarrow 0.$
\end{theorem}
\begin{proof}
 Our objective is to establish that
\[
\mathbb{P}(\hat{k}_\delta = \infty) = 0 \quad \text{for } \|y_i^\delta - y_i^\dagger\| > 0,\quad i = 0, 1, \ldots, P-1.
\]
By the definition of \(\hat{k}_\delta\), the event \(\Psi := \{ \hat{k}_\delta = \infty \}\) can be equivalently expressed as
\[ \Psi = \left\{ \|F_{i}(u_k^\delta) - y_i^\delta\| + \frac{M}{(k+a)} > \tau \|y_i^\delta - y_i^\dagger\| ; i=0, 1, \ldots P-1 ; \quad \text{for all integer } k \geq 0 \right\}.\]
We now establish the desired result.  
Recall from \eqref{solved in Lemma 1}, \eqref{disc used in lemma 1}, and \eqref{condition for first lemma} that for 
$\lambda_k \in [0,\tfrac12)$ and $\omega_k \in [\omega,\Omega]$, we have
\begin{align}
\nonumber
\|u_{k+1}^\delta - \hat{u}\|^2 
+ D\|y_{i_k}^\delta - F_{i_k}(u_{k}^\delta)\|^2 
& \leq (1 - \lambda_k)^2 \|u_{k}^\delta -\hat{u}\|^2 
+ \lambda_k^2 \big(1 + \kappa^{-2}\omega_k^2 L^2\big)\rho^2 \\
\nonumber
&\quad + 2 \lambda_k (1 - \lambda_k) \|u_{k}^\delta - \hat{u}\|\,\rho  
+ \frac{M\omega_k(1 - \lambda_k)(1+ \eta)}{\tau(k + a)^2}.
\end{align}

Since $(1-\lambda_k)^2 \le 1$ and $\|u_k^\delta - \hat{u}\| \le c(\rho)$ by Lemma~\ref{first lemma}, the above yields
\begin{align}\label{ineq_after_simplification}
\nonumber
\|u_{k+1}^\delta - \hat{u}\|^2 
+ D\|y_{i_k}^\delta - F_{i_k}(u_{k}^\delta)\|^2
& \le \|u_{k}^\delta -\hat{u}\|^2 
+ \lambda_k^2 \big(1 + \kappa^{-2}\omega_k^2 L^2\big)\rho^2 \\
&\quad + 2\lambda_k \rho\, c(\rho) 
+ \frac{M\Omega(1 - \lambda_k)(1+ \eta)}{\tau(k + a)^2}.
\end{align}

Next, using $\lambda_k^2 \le \lambda_k$ and $\omega_k \le \Omega$, and noting that 
$u_k^\delta$ is $\mathcal{F}_k$–measurable while $i_k$ is chosen uniformly from $\{0,1, \dots,P-1\}$ independent of $\mathcal{F}_k$, we take conditional expectation with respect to $\mathcal{F}_k$.  
By the independence of $i_k$, the expected residual term becomes the average over all $P$ components.  
Thus \eqref{ineq_after_simplification} implies
\begin{align}\label{eqn_of_lemma1_part2}
\nonumber
&\mathbb{E}\big[\|u_{k+1}^\delta - \hat{u}\|^2 \,\big|\, \mathcal{F}_k\big] 
+ \frac{D}{P} \sum_{i=0}^{P-1}\|y_i^\delta - F_i(u_{k}^\delta)\|^2  \\
&\quad \leq \mathbb{E}\big[\|u_{k}^\delta -\hat{u}\|^2 \,\big|\, \mathcal{F}_k\big] 
+ \rho^2\big(1 + \kappa^{-2}\Omega^2 L^2\big)\lambda_k 
+ 2\rho\, c(\rho)\,\lambda_k 
+ \frac{M\Omega(1 - \lambda_k)(1+ \eta)}{\tau(k + a)^2}.
\end{align}

We now take full expectation in \eqref{eqn_of_lemma1_part2} and sum from $k=0$ to $t-1$.  
The left–hand side telescopes in the first term, while the residual terms remain additive.  
Using $\mathbb{E}[\|u_t^\delta - \hat{u}\|^2] \ge 0$, we arrive at
\begin{equation}\label{l0_readable}
\begin{split}
\sum_{k=0}^{t -1} \mathbb{E}\left[ \frac{1}{P}\sum_{i=0}^{P-1} \|y_i^\delta - F_i(u^\delta_k)\|^2 \right] 
&\leq  \frac{1}{D}\bigg( \mathbb{E}[\|u_0^\delta-\hat u\|^2] \\
&\quad +\big(\rho^2(1 + \kappa^{-2}\Omega^2 L^2) + 2\rho\, c(\rho)\big) \sum_{k=0}^{t-1} \lambda_k \\
&\quad + \frac{M\Omega(1+ \eta)}{\tau} \sum_{k=0}^{t-1}\frac{1 - \lambda_k}{(k + a)^2} \bigg).
\end{split}
\end{equation}

Finally, recalling that the initial guess $u_0^\delta$ is deterministic (or $\mathcal F_0$--measurable) and satisfies
$\|u_0^\delta-\hat u\|\le\rho$. Hence
\(\mathbb{E}[\|u_0^\delta-\hat u\|^2]=\|u_0^\delta-\hat u\|^2\le\rho^2\).
 So the bound \eqref{l0_readable} simplifies to
\begin{equation}\label{l0}
\begin{split}
\sum_{k=0}^{t -1} \mathbb{E}\left[ \frac{1}{P}\sum_{i=0}^{P-1} \|y_i^\delta - F_i(u^\delta_k)\|^2 \right] 
&\leq  \frac{1}{D}\bigg( \rho^2 
+\big(\rho^2(1 + \kappa^{-2}\Omega^2 L^2) + 2\rho\, c(\rho)\big) \sum_{k=0}^{t-1} \lambda_k \\
&\quad + \frac{M\Omega(1+ \eta)}{\tau} \sum_{k=0}^{t-1}\frac{1 - \lambda_k}{(k + a)^2} \bigg).
\end{split}
\end{equation}

By applying \ref{cond 2 for first lemma} and the result $\sum_{k=0}^{\infty} \frac{1}{(k + a)^2} \leq \sum_{k=1}^{\infty} \frac{1}{k^2}  = \frac{\pi^2}{6}$ for $a \geq 1$ in \eqref{l0}, we  further get
\begin{equation}\label{sum fini}
 \sum_{k=0}^{t -1} \mathbb{E}\left[  \|y_{i_k}^\delta - F_{i_k}(u^\delta_k)\|^2 \right] =  \sum_{k=0}^{t -1} \mathbb{E}\left[\frac{1}{P} \sum_{i=0}^{P-1} \|y_i^\delta - F_i(u^\delta_k)\|^2 \right] < \infty.
\end{equation}
Now, define \(\chi_\Psi\) as the characteristic function associated with the set \(\Psi\), such that
\[
\chi_\Psi(w) =
\begin{cases}
1, & \text{if } w \in \Psi, \\
0, & \text{if } w \notin \Psi.
\end{cases}
\]
By definition of $\Psi$, for every integer $k\ge0$ we have, for all $i=0, 1, \dots,P-1$,
\[
\|y_i^\delta - F_i(u_k^\delta)\| > \tau\|y_i^\delta-y_i^\dagger\| - \frac{M}{k+a}.
\]
Using the elementary inequality $(l-m)^2\ge \tfrac12 l^2 - m^2$ with
$l=\tau\|y_i^\delta-y_i^\dagger\|$ and $m=\dfrac{M}{k+a}$, we obtain for each $i$ on $\Psi$
\[
\|y_i^\delta - F_i(u_k^\delta)\|^2 
\ge \tfrac12\tau^2\|y_i^\delta-y_i^\dagger\|^2 - \frac{M^2}{(k+a)^2}.
\]
Multiplying by the indicator $\chi_\Psi$ and taking expectation, then averaging over the uniformly sampled index $i_k$, gives
\begin{equation*}
\begin{split}
\mathbb{E}\big[\|y_{i_k}^\delta - F_{i_k}(u_k^\delta)\|^2\big]
&\ge \mathbb{E}\big[\|y_{i_k}^\delta - F_{i_k}(u_k^\delta)\|^2\chi_\Psi\big] \\
&= \mathbb{E}\Big[\frac{1}{P}\sum_{i=0}^{P-1}\|y_i^\delta - F_i(u_k^\delta)\|^2\chi_\Psi\Big]\\
&\ge \frac{\tau^2}{2P}\,\mathbb{E}\Big[\sum_{i=0}^{P-1}\|y_i^\delta-y_i^\dagger\|^2\chi_\Psi\Big]
- \frac{M^2}{(k+a)^2}\,\mathbb{E}[\chi_\Psi] \\
&\ge \frac{\tau^2}{2P} \sum_{i=0}^{P-1}\|y_i^\delta-y_i^\dagger\|^2\,\mathbb{P}(\Psi) - \frac{M^2}{(k+a)^2}\,\mathbb{P}(\Psi).
\end{split}
\end{equation*}


Sum the last inequality over $k=0,\dots,t-1$ and use linearity of expectation, we arrive at 
\begin{equation}\label{per-k-tower}
\sum_{k=0}^{t-1}\mathbb{E}\big[\|y_{i_k}^\delta - F_{i_k}(u_k^\delta)\|^2\big]
\ge \frac{\tau^2}{2P} \sum_{i=0}^{P-1}\|y_i^\delta-y_i^\dagger\|^2t\,\mathbb{P}(\Psi)
- M^2\,\mathbb{P}(\Psi)\sum_{k=0}^{t-1}\frac{1}{(k+a)^2}.
\end{equation}
By \eqref{sum fini} there exists a constant $\mathfrak{C}>0$ independent of $t$ such that
\[
\sum_{k=0}^{t-1}\mathbb{E}\big[\|y_{i_k}^\delta - F_{i_k}(u_k^\delta)\|^2\big] \le \mathfrak{C}.
\]
Combining this with \eqref{per-k-tower} gives
\[
\mathfrak{C} \ge \frac{\tau^2}{2P} \sum_{i=0}^{P-1}\|y_i^\delta-y_i^\dagger\|^2\, t\,\mathbb{P}(\Psi)
- M^2\,\mathbb{P}(\Psi)\sum_{k=0}^{t-1}\frac{1}{(k+a)^2}.
\]
Rearrange to obtain
\[
\mathbb{P}(\Psi)\left(\frac{\tau^2}{2P} \sum_{i=0}^{P-1}\|y_i^\delta-y_i^\dagger\|^2\, t - M^2\sum_{k=0}^{t-1}\frac{1}{(k+a)^2}\right) \le \mathfrak{C}.
\]
The series  $\sum_{k=0}^{\infty} \frac{1}{(k + a)^2} \leq \sum_{k=1}^{\infty} \frac{1}{k^2}  = \frac{\pi^2}{6}$, hence
\[
\mathbb{P}(\Psi)\left(\frac{\tau^2}{2P} \sum_{i=0}^{P-1}\|y_i^\delta-y_i^\dagger\|^2\, t - \frac{M^2 \pi^2}{6}\right) \le \mathfrak{C}.
\]
Letting \(t \to \infty\), we observe that the term in parentheses on the left--hand side grows without bound whenever \(\mathbb{P}(\Psi) > 0\). This would contradict the boundedness of the left--hand side unless \(\mathbb{P}(\Psi) = 0\). Therefore, we conclude that  
\[
\mathbb{P}(\Psi) = 0, \quad\text{i.e.,}\quad \mathbb{P}(\hat{k}_\delta = \infty) = 0.
\]  

This implies that the complement of the event \(\Psi\) occurs with probability one, i.e., along almost every sample path, \eqref{newdiscrepancy} holds for some finite \(\hat{k}_\delta\). 
Finally, from the definition of $\hat{k}_\delta$ in \ref{newdiscrepancy}, we get
 \[\frac{M}{(\hat{k}_\delta + a)} \leq \tau  \|y_i^\delta - y_i^\dagger\| \leq \tau \| y^\delta - y^\dagger \| \to 0 \text{\quad as \quad} \delta \to 0.\]
 Consequently, we must have $\hat{k}_\delta \to \infty$ as $\delta \to 0$. This completes the proof.
\end{proof}
The proof of the following lemma draws inspiration from the methodologies outlined in \cite{Zhang and Jin Q} and \cite{real2024hanke}. In this result, we examine the asymptotic behavior of the stopping index prescribed by Rule~\ref{Hanke Rule} when applied to noisy data satisfying Assumption~\ref{new assumption}. It is noteworthy that the proof of this lemma relies essentially on Theorem~\ref{kdelta hat finite}.
\begin{lemma} \label{secound lemma}
Let the sequence $\{\lambda_k\}$, $\omega$ and $\Omega$ be chosen such that $\ref{condition for first lemma}$ and $\ref{cond 2 for first lemma}$ hold and let Assumption $\ref{common assumptions}$ hold. Additionally, let $\{y^\delta\}$ represent a set of noisy data that fulfill Assumption $\ref{new assumption}$. Furthermore, assume that $k_* := k_*(y^\delta)$ is the stopping index as defined in Rule \ref{Hanke Rule}. Then 
\[ \mathbb{E} \left[\Psi(k_*(y^\delta), y^\delta ) \right] \to 0 \quad \text{as} \quad \delta \to 0.\]
Thus, we have 
\[
\mathbb{E} \left[\| F(u^\delta_{k_*(y^\delta)}) - y^\delta \|\right] \to 0 \quad \text{and} \quad \mathbb{E} \left[ k_*(y^\delta) \| y^\delta - y^\dag \|^2 \right] \to 0 \quad \text{as} \quad \delta \to 0.
\]
\end{lemma}

\begin{proof}
By employing the condition \ref{cond 2 for first lemma} and the fact that  $\sum_{k=0}^{\infty} \frac{1}{(k + a)^2} \leq \sum_{k=1}^{\infty} \frac{1}{k^2} = \frac{\pi^2}{6}$ for $a \geq 1$ in  \ref{l0}, we arrive at
 \begin{equation}\label{II eqn of lemma 2}
\sum_{k=0}^{\hat{k}_\delta -1} \mathbb{E}\left[ \sum_{i=0}^{P-1} \|y_i^\delta - F_i(u^\delta_k)\|^2 \right]  \leq  PD^{-1}C_1,
\end{equation}
\ where $C_1= \left(\rho^2   + \left(\rho^2(1 + \kappa^{-2}\Omega^2 L^2) + 2\rho c(\rho)\right)\lambda  + \frac{M\Omega\pi^2 (1+ \eta)}{6\tau}\right)$.
By minimality  of $\Psi(k_*, y^\delta)$, it follows from Rule \ref{Hanke Rule} that
\[ \sum_{i=0}^{P-1}\|F_i(u_k^\delta) - y_i^\delta\|^2 = \frac{\Psi(k, y^\delta)}{k+a} \geq \frac{\Psi(k_{*}(y^\delta), y^\delta)}{k+a}. \] 
This and \ref{II eqn of lemma 2} further yields
\begin{equation}\label{1*}
   \mathbb{E}\left[  \Psi (k_*(y^\delta), y^\delta) \right]\sum^{\hat{k}_\delta - 1}_{k=0} \frac{1}{k+a} \leq \sum_{k=0}^{\hat{k}_\delta -1} \mathbb{E}\left[ \sum_{i=0}^{P-1} \|y_i^\delta - F_i(u^\delta_k)\|^2 \right]  \leq  PD^{-1}C_1.\end{equation}
Note that \[
\sum_{k=0}^{\hat{k}_\delta -1} \frac{1}{k+a} \geq  \sum_{k=0}^{\hat{k}_\delta -1}\int_{k}^{k+1}\frac{1}{t+a}dt = \int_{0}^{\hat{k}_\delta}\frac{1}{t+a}dt = \log \frac{\hat{k}_\delta + a}{a}.
\]
Plugging the last estimate in \ref{1*} yields
\[\bigg(\log\frac{\hat{k}_\delta +a}{a}\bigg) \mathbb{E} \left [ \Psi(k_* (y^\delta), y^\delta) \right ] \leq P D^{-1} C_1 . \]
Since $\hat{k}_\delta \to \infty$ as $\delta \to 0$, we must have $ \mathbb{E}\left[\Psi(k_*(y^\delta), y^\delta)\right] \to 0$ as $\delta \to 0$. Further, we note that
\[
 a \sum_{i=0}^{P-1}\| F_i(u^\delta_{k_*(y^\delta)}) - y_i^\delta \|^2 =a\| F(u^\delta_{k_*(y^\delta)}) - y^\delta \|^2 \leq \Psi(k_*(y^\delta), y^\delta).\]
Also, Assumption \ref{new assumption} implies that
\[
(k_*(y^\delta) +a) \varkappa^2 \| y^\delta - y^\dagger \|^2 \leq  \Psi(k_*(y^\delta), y^\delta).
\]
Hence, we conclude that 
$ \mathbb{E}\left[\|F (u_{k_{*}(y^\delta)}^\delta) - y^\delta\|\right] \to 0$ and $\mathbb{E}\left[ k_*(y^\delta) \| y^\delta - y^\dagger \|^2 \right] \to 0$ as $\delta \to 0,$ since $\mathbb{E}\left[ \Psi(k_*(y^\delta), y^\delta)\right] \to 0$ as $\delta \to 0$. Thus, result.
\end{proof}

The next result demonstrates that the iterates $u_k^\delta$ generated by Algorithm~\ref{alg:IRSGD Algorithm} remains in a ball around $u_0$ 
 for any $0 \leq k < k_*$, where $k_*$ is defined via Rule \ref{Hanke Rule} (note that Lemma \ref{first lemma} shows the same for $0 \leq k < \hat{k_{\delta}}$).
\begin{lemma}\label{3.3}
Let all the conditions of Lemma $\ref{secound lemma}$ be satisfied. Additionally, assume that 
\begin{equation}\label{ass of lemma 3}
 D_1 :=   \kappa^2 - \Omega(1 - \lambda_{\text{max}})(1 - 3\eta) + \omega^2 L^2 > 0.                                                         
\end{equation}
Then, if $u^\delta_k \in B_{c(\rho)}(\hat{u})$ for any $0 \leq k < k_*$, it follows that $u^\delta_{k+1} \in B_{c(\rho)}(\hat{u})\subset B_{\rho +c(\rho)}(u_0)$ almost surely, provided $\delta$ is sufficiently small.
\end{lemma}
\begin{proof}
For $k=0$, result is trivial. 
 By assuming $u_k^\delta \in B_{c(\rho)}(\hat{u})$  for any $0 \leq k < k^*,$ we show that $u_{k+1}^\delta \in B_{c(\rho)}(\hat{u})$ almost surely. From \ref{solved in Lemma 1} and the inequalities \(\|u_{k}^\delta - \hat{u}\| \leq c(\rho)\) and  \(\|u_0 - \hat{u}\| \leq \rho\),  we get
    \begin{align}
        \|u_{k+1}^\delta\ - \hat{u}\|^2 &-  \|u_{k}^\delta\ - \hat{u}\|^2\nonumber  \\
        & \leq \rho^2 (1 + \kappa^{-2}\Omega^2 L^2)\lambda_k^2  +  2\omega_k(1-\lambda_k)(1+\eta)\|F_{i_k}(u_k^\delta) - y_{i_k}^\delta\|\|y_{i_k}^\dagger -y_{i_k}^\delta\| \\ \nonumber 
        & \quad + 2 \lambda_k(1-\lambda_k)\rho c(\rho)  - (2\omega_k(1-\lambda_k)(1 - \eta) - \omega_k^2 L^2 - \kappa^2)\|F_{i_k}(u_k^\delta) -y_{i_k}^\delta\|^2 \\ \nonumber
        &\leq \rho^2 (1 + \kappa^{-2}\Omega^2 L^2)\lambda_k^2 + \omega_k(1-\lambda_k)(1+\eta)\|y_{i_k}^\dagger - y_{i_k}^\delta\|^2 + 2 \lambda_k(1-\lambda_k)\rho c(\rho) \\ \nonumber 
        & \quad + ( \omega_k(1 -\lambda_k)(1 + \eta)- \left(2\omega_k(1-\lambda_k)(1 - \eta) - \omega_k^2 L^2- \kappa^2)\right)\|F_{i_k}(u_k^\delta) -y_{i_k}^\delta\|^2.
    \end{align}
   Using  the fact that $1-\lambda_k \leq 1$ and $ \lambda_k^2 \leq \lambda_k$ for $\lambda_k \in [0, \frac{1}{2})$, we get     
    \begin{equation}
       \begin{aligned}
            \|u_{k+1}^\delta\ - \hat{u}\|^2  &-  \|u_{k}^\delta\ - \hat{u}\|^2 \leq  (\rho^2(1 + \kappa^{-2}\Omega^2 L^2) + 2\rho c(\rho))\lambda_k  + \omega_k(1-\lambda_k)(1+\eta)\|y_{i_k}^\dagger - y_{i_k}^\delta\|^2 \\ 
        & + ( \omega_k(1 -\lambda_k)(1 + \eta)- \left(2\omega_k(1-\lambda_k)(1 - \eta) - \omega_k^2 L^2- \kappa^2)\right)\|F_{i_k}(u_k^\delta) -y_{i_k}^\delta\|^2.
       \end{aligned} 
    \end{equation} 
By taking the total conditional and utilizing $1-\lambda_k \leq 1$, we further reach at
\begin{align}
       \mathbb{E}[ \|u_{k+1}^\delta - \hat{u}\|^2] -  \mathbb{E}[\|u_{k}^\delta - \hat{u}\|^2]\nonumber &\leq  (\rho^2(1 + \kappa^{-2}\Omega^2 L^2) + 2\rho c(\rho))\lambda_k +\Omega(1+\eta)\mathbb{E}[\|y^\dagger - y^\delta\|^2] \\ \nonumber 
        & \hspace{10mm} - D_1 \mathbb{E}[\|F(u_k^\delta) -y^\delta\|^2].
    \end{align}
   The last inequality derives that
    \begin{equation}\label{sum for corollory}
    \begin{split}
         \mathbb{E}[\|u_{k+1}^\delta - \hat{u}\|^2] + D_1 \sum_{l=0}^{k} \mathbb{E}[\|F(u_l^\delta) - y^\delta\|^2] &\leq  \rho^2 +  ( \rho^2(1 + \kappa^{-2}\Omega^2 L^2) + 2\rho c(\rho))\sum_{l=0}^{k}\lambda_l \\
         & \hspace{10mm} +\Omega(1+\eta)(k+1) \mathbb{E}[\|y^\delta - y^\dagger\|^2]
    \end{split} 
    \end{equation}
    and hence by using \ref{cond 2 for first lemma} and the fact that $k< k_*$, where $k_*$ is a finite integer, we get
    \[ \mathbb{E}[\|u_{k+1}^\delta - \hat{u}\|^2] \leq  \rho^2 +   (\rho^2(1 + \kappa^{-2}\Omega^2 L^2) + 2\rho c(\rho))\lambda + \Omega(1+\eta) \mathbb{E}[k_{*} \|y^\delta - y^\dagger\|^2].\]
    From Lemma \ref{secound lemma},  one can note that for sufficiently small $\delta$, there exist a constant $\xi$ such that
    \begin{equation}\label{zeta condn}
        \Omega(1+\eta)\mathbb{E}[ k_* (y^\delta) \|y^\delta - y^\dagger\|^2] \leq \xi,\ 
\text{with}  \  \rho^2 +  \ (\rho^2(1 + \kappa^{-2}\Omega^2 L^2) + 2\rho c(\rho))\lambda  + \xi \leq c(\rho)^2.\end{equation}
    Hence, as a result, we get
\( \mathbb{E} [\|u_{k+1}^\delta\ - \hat{u}\|^2] \leq c(\rho)^2.\)
 By incorporating Jensen's inequality with $f(x) = x^2$, we conclude that
\[\mathbb{E}[\|u^\delta_{k+1} - \hat{u}\|] \leq \mathbb{E}[\|u^\delta_{k+1} - \hat{u}\|^2]^{1/2} \leq c(\rho).\]
 This concludes the proof.
\end{proof}
\section{Convergence Analysis}\label{sec-convo}
In this section, we prove that \(u^\delta_{k_*}\) converges almost surely to a solution of \ref{ststart} as \(\delta \to 0\). We first analyze the exact-data analogue of Algorithm~\ref{alg:IRSGD Algorithm} and establish its almost-sure convergence. Next, we show a stability property along each sample path that connects the noisy- and exact-data cases. These results together yield the main convergence theorem. The exact-data version of Algorithm~\ref{alg:IRSGD Algorithm} is given below.

\begin{algorithm}
\caption{Exact-data IRSGD }
\label{alg:Exact-data IRSGD}
\begin{algorithmic}[1]
\STATE{\textbf{Given:} $F_i, y_i^\dagger$ for all $i \in \{0, 1, \ldots, P-1\}$ }
\STATE{\textbf{Initialize:} $u_0 = u^{(0)}$}
\FOR{$k \geq 0$}
    \STATE{Choose $\omega_k \in [\omega, \Omega]$ and the weight parameter $\lambda_k$ such that
    \[
    0 \leq \lambda_k \leq \lambda_{\text{max}} < \frac{1}{2}.
    \]}
    \STATE{Randomly pick an index $i_k$ from $\{0, 1, \ldots, P-1\}$ uniformly.}
    \STATE{Update
    \[
    u_{k+1} = u_k - \omega_k F'_{i_k}(u_k)^* (F_{i_k}(u_k) - y_{i_k}^\dagger) - \lambda_k (u_k - u^{(0)}),
    \]
    where $F'_{i_k}$ is the Fréchet derivative of $F_{i_k}$, and $u^{(0)}$ is the initial guess.}
\ENDFOR
\end{algorithmic}
\end{algorithm}

The result below identifies an event \(\daleth\) of probability one, consisting of sample paths that will serve as the basis for the convergence analysis of the IRSGD method in the exact-data setting presented in Algorithm~\ref{alg:Exact-data IRSGD}. 
Specifically, we define
\begin{equation}\label{event daleth}
    \daleth = \left\{ \sum_{l=0}^{\infty} \sum_{i=0}^{P-1} \|F_i(u_l) - y_i^\dagger\|^2 < \infty \right\},
\end{equation}
and show that this event occurs almost surely, i.e., \(\mathbb{P}(\daleth) = 1\).

\begin{cor}\label{for sum} Let the assumptions of Lemma $\ref{first lemma}$ be satisfied. Additionally, assume that IRSGD method $\ref{SIRLI 2}$ is stopped according to Rule $\ref{Hanke Rule}$, then for sufficiently small $\delta$, we have

\[\sum_{l=0}^{k_* -1} \mathbb{E}[\|F(u_l^\delta) - y^\delta\|^2] \leq D_{1}^{-1} \left( (\rho^2(1 + \kappa^{-2}\Omega^2 L^2) + 2\rho c(\rho))\sum_{l=0}^{k_* -1}\lambda_k +\rho^2 + \xi\right) < \infty,\]
where $D_1$ is defined  in $\ref{ass of lemma 3}$. In particular, for the case $y^\delta = y^\dagger$, we have 
\[ \sum_{l=0}^{\infty} \mathbb{E}[\|F(u_l) - y^\dagger\|^2] \leq D_{1}^{-1}  \left( \rho^2 +  (\rho^2(1 + \kappa^{-2}\Omega^2 L^2) + 2\rho c(\rho))\sum_{l=0}^{\infty}\lambda_k \right) < \infty.\]
It follows from the above inequality that the event \(\daleth\) occurs with probability one, i.e., \(\mathbb{P}(\daleth) = 1\).
\end{cor}
\begin{proof}
    
   From \ref{sum for corollory}, we get 
    \begin{equation*}
        D_1 \sum_{l=0}^{k_* -1} \mathbb{E}[\|F(u_l^\delta) - y^\delta\|^2] \leq \rho^2 +   (\rho^2(1 + \kappa^{-2}\Omega^2 L^2) + 2\rho c(\rho))\sum_{l=0}^{k_* -1} \lambda_l  + \Omega(1+ \eta)  \mathbb{E}[k_*\|y^\delta - y^\dagger\|^2].
    \end{equation*}
 Using \ref{zeta condn}, we note that
    \begin{equation*}\label{sum for corollory in 1(a)}
         \sum_{l=0}^{k_* -1} \mathbb{E}[\|F(u_l^\delta) - y^\delta\|^2] \leq D_{1}^{-1} \left( (\rho^2(1 + \kappa^{-2}\Omega^2 L^2) + 2\rho c(\rho))\sum_{l=0}^{k_* -1}\lambda_k +\rho^2 + \xi\right) < \infty.
    \end{equation*}
    Thus, result. The case $y^{\delta}=y^\dagger$ can be handled similarly.
\end{proof}
\begin{rema}
   If \( \lambda_{\text{max}} = 0 \), the parameter \( \kappa \) in conditions \ref{condition for first lemma} and \ref{ass of lemma 3} can be set to zero. This adjustment mirrors the conditions required to establish the same results for SGD method \ref{SGD}.
\end{rema}

\subsection{Convergence for exact-data case}\label{secfornonnoisy}
In this subsection, we establish the convergence of Algorithm~\ref{alg:Exact-data IRSGD}. Let us assume that
      \begin{equation}\label{condn for lemma 4}
      D_2 := (2\omega(1 - \lambda_{\text{max}})(1-\eta) - \Omega^2 L^2 - \kappa^2) > 0.
      \end{equation}

      \begin{lemma}\label{uk convergent}
      Suppose all the conditions of Lemma $\ref{first lemma}$ are satisfied. Additionally, assume that \eqref{condn for lemma 4} holds.
    Then, the sequence of iterates $\{u_k\}_{k\geq 1}$ produced by Algorithm~\ref{alg:Exact-data IRSGD} is a Cauchy sequence almost surely.

\end{lemma} 

\begin{proof}
 
 Firstly, we define $e_k := u_k - u^{\dagger}$, for a solution $u^\dagger$ of \ref{ststart}. As shown in Lemma \ref{first lemma}, $\mathbb{E}[\| e_k\|^2]$ is a bounded sequence and hence it has a convergent subsequence $\mathbb{E}[\| e_{k_n}\|^2]$. Let it converge to some $\epsilon \geq 0$. It follows directly from \ref{solved in Lemma 1} that, for $\delta = 0$ we have 
 \begin{align} \label{solved in Lemma 3}
\|u_{k+1} - u^{\dagger}\|^2 \nonumber &\leq  (1 - \lambda_k)^2 \|u_k - u^{\dagger}\|^2 + \lambda_k^2 (1 + \kappa^{-2}\omega_k^2 L^2) \|u_0 - u^{\dagger}\|^2 \\ \nonumber
& + 2 \lambda_k (1 - \lambda_k) \|u_k - u^{\dagger}\| \|u_0 - u^{\dagger}\| \\
&- (2\omega_k(1 - \lambda_k)(1-\eta) - \omega_k^2 L^2 - \kappa^2)\|y_{i_k} - F_{i_k}(u_k)\|^2.  \nonumber
\end{align}
From the definition of $F$ and the measurability of $u_k^{\delta}$ concerning $\mathcal{F}_k$, we deduce that
\begin{align} 
\mathbb{E}[\|u_{k+1} - u^{\dagger}\|^2 | \mathcal{F}_k] \nonumber &\leq \mathbb{E}[ (1 - \lambda_k)^2 \|u_k - u^{\dagger}\|^2  + \lambda_k^2 (1 + \kappa^{-2}\omega_k^2L^2)\|u_0 - u^{\dagger}\|^2  \\ \nonumber
& \quad + 2 \lambda_k (1 - \lambda_k)\|u_k - u^{\dagger}\|  \|u_0 - u^{\dagger}\| | \mathcal{F}_k]\\
& \quad - \frac{1}{P}\|y - F(u_k)\|^2 (2\omega(1 - \lambda_k)(1-\eta) - \Omega^2 L^2 - \kappa^2). \nonumber
\end{align}
Taking the total conditional and using \ref{condn for lemma 4}, we get
 \begin{align} 
\mathbb{E}[\|u_{k+1} - u^{\dagger}\|^2 ] + \frac{D_2}{P} \mathbb{E}[\|y - F(u_k)\|^2] \nonumber &\leq  (1 - \lambda_k)^2 \mathbb{E}[\|u_k - u^{\dagger}\|^2 ] + \lambda_k^2 (1 + \kappa^{-2}\Omega^2L^2) \|u_0 - u^{\dagger}\|^2 \\ \nonumber
& + 2 \lambda_k (1 - \lambda_k) \mathbb{E}[\|u_k - u^{\dagger}\| ]\|u_0 - u^{\dagger}\|.
\end{align} 
This means that
\[\mathbb{E}[\|u_{k+1} - u^{\dagger}\|^2 ]^{1/2} \leq  (1 - \lambda_k) \mathbb{E}[\|u_k - u^{\dagger}\|^2 ]^{1/2} + \rho\lambda_k \sqrt{1 + \kappa^{-2}\Omega^2 L^2}.\]
Hence by induction, we further get that
\begin{equation}\label{mid eqn first theorem}
    \mathbb{E}[\|e_k\|^2]^{1/2} \leq \mathbb{E}[\|e_m \|^2]^{1/2} \prod_{l=m}^{k-1}(1 - \lambda_l) + \rho \sqrt{1 + \kappa^{-2}\Omega^2 L^2}(1 - \prod_{l=m}^{k-1}(1-\lambda_l))
\end{equation}
for $m < k,$ where we have used the following result
\[1 - \prod_{l=m}^{n}(1 - \lambda_l) = \sum_{j =m}^{n} \lambda_j \prod_{l=j+1}^{n}(1-\lambda_l) \]
which, under the standard condition that $ \prod_{l = n+1}^{n} (1 - \lambda_l) =1$, is satisfied for $n \geq m$. \\
If \( k_{n-1} < k < k_n \), then according to \ref{mid eqn first theorem} it follows that
\begin{align*}
    \mathbb{E}[\|e_k\|^2]^{1/2} &\leq \mathbb{E}[\|e_{k_{n-1}}\|^2]^{1/2} + ( \rho\sqrt{1+\kappa^{-2}\Omega^2 L^2} - \mathbb{E}[\|e_{k_{n-1}}\|^2]^{1/2} )\left(1- \prod_{l=k_{n-1}}^{k-1} (1 - \lambda_l)\right), \\
     \mathbb{E}[\|e_k\|^2 ]^{1/2} &\geq \mathbb{E}[\|e_{k_{n}}\|^2]^{1/2} - ( \rho\sqrt{1+\kappa^{-2}\Omega^2 L^2} - \mathbb{E}[\|e_{k_{n}}\|^2]^{1/2} )\left(1- \prod_{l=k}^{k_{n-1}} (1 - \lambda_l)\right).
\end{align*}
These bounds, along with the convergence of \( \mathbb{E}[\|e_{k_n}\|^2] \), directly imply
\[ \lim_{k \to \infty} \mathbb{E}[\|e_k\|^2]^{1/2} = \epsilon^{1/2}, \]
since \ref{cond 2 for first lemma} implies that
\begin{equation}\label{prod beta in thm 1}
    \prod_{l=m}^{k-1} (1 - \lambda_l) \to 1 \text{ as } m \to \infty, \; k > m.
\end{equation}
Note that the operator \(\mathbb{E}[\langle \cdot, \cdot \rangle]\) defines an inner product. Now, for any \(n \geq k\), we pick an index \(m\) such that \(k \leq m \leq n\) and
\begin{equation}\label{choiceofl}
    \mathbb{E}[\| y - F(u_m)\|^2] \leq \mathbb{E}[\| y - F(u_i)\|^2] \quad \text{for all } k \leq i \leq n.
\end{equation}
Using the triangular inequality
\begin{equation}\label{tringular}
   \mathbb{E}[\| e_n - e_k\|^2]^{1/2} \leq \mathbb{E}[\| e_n - e_m\|^2]^{1/2} + \mathbb{E}[\| e_m - e_k\|^2]^{1/2}, 
\end{equation}
along with the identities
\begin{align}\label{enem}
    \mathbb{E}[\| e_n - e_m\|^{2}] &= 2\mathbb{E}[\langle e_m - e_n , e_m \rangle] + \mathbb{E}[\| e_n\|^{2}] - \mathbb{E}[\| e_m\|^{2}], \\  \label{emek}
    \mathbb{E}[\| e_m - e_k\|^{2}] &= 2\mathbb{E}[\langle e_m - e_k, e_m \rangle] + \mathbb{E}[\| e_k\|^{2}] - \mathbb{E}[\| e_m\|^{2}],
\end{align}
we aim to prove the desired result. To that end, it suffices to show that both \(\mathbb{E}[\| e_n - e_m\|^2]\) and \(\mathbb{E}[\| e_m - e_k\|^2]\) converge to zero as \(k \to \infty\).

Since \(\mathbb{E}[\| e_k \|^2] \to \epsilon\) as \(k \to \infty\), the final two terms on the right-hand sides of the above identities vanish in the limit, i.e., \(\epsilon - \epsilon = 0\). Thus, it remains to show that \(\mathbb{E}[\langle e_m - e_k, e_m \rangle]\) also tends to zero as \(k \to \infty\). The same argument applies to \(\mathbb{E}[\langle e_m - e_n, e_m \rangle]\).

To show this, we recall the recursive structure of \(u_k\), from which we obtain
\[
e_k = e_m \prod_{l=m}^{k-1} (1 - \lambda_l) + \sum_{j=m}^{k-1} \omega_j F_{i_j}^\prime(u_{j})^\ast (y_{i_j} - F_{i_j}(u_j)) \prod_{l=j+1}^{k-1} (1 - \lambda_l) + (u_0 - u^\dag) \prod_{l=m}^{k-1} (1 - \lambda_l),
\]
for \(m < k\). This formula allows us to determine that
\begin{align*}
     |\mathbb{E}[\langle e_m - e_n, e_m \rangle]| &\leq \left( 1- \prod_{l=m}^{n-1} (1 - \lambda_l) \right) |\mathbb{E}[\langle u_0 - u_m, e_m \rangle]|\\
     & \quad + \sum_{j=m}^{n-1}\omega_j \left( \prod_{l=j+1}^{n-1} (1 - \lambda_l) \right) |\mathbb{E}[\langle y_{i_j} - F_{i_j}(u_j), F_{i_j}^\prime(u_j)(u_m - u^\dag) \rangle], \\
     |\mathbb{E}[\langle e_m - e_n, e_m \rangle]| &\leq \left( 1- \prod_{l=m}^{n-1} (1 - \lambda_l) \right) |\mathbb{E}[\langle u_0 - u_m, e_m \rangle]|\\
     & \quad + \sum_{j=m}^{n-1}\omega_j \left( \prod_{l=j+1}^{n-1} (1 - \lambda_l) \right) |\mathbb{E}[\langle y - F(u_j), F^\prime(u_j)(u_m - u^\dagger) \rangle] |.
\end{align*}
Due to  \ref{prod beta in thm 1}, the first term on the right-hand side approaches zero because \( |\mathbb{E}[ \langle u_0 - u_m, e_m \rangle ]| \) is bounded. In addition to that, by using the inequality \( 1-\lambda_k \leq 1 \) along with  \ref{cond 2 for first lemma}, one can estimate the second term on the right-hand side as
\begin{equation}\label{neww}
     \Omega(1 + 3\eta) \sum_{j=m}^{n-1} \mathbb{E}[ \|y - F(u_j)\|^2]. \end{equation}
Now from Corollary \ref{for sum}, it follows that \ref{neww} approaches zero as \( k \to \infty \). This means \( \mathbb{E}[\langle e_m - e_n, e_m \rangle] \to 0 \). Analogously,  we  have \( \mathbb{E}[ \langle e_m - e_k, e_m \rangle] \to 0 \).  Hence from \eqref{enem} and \eqref{emek} we have 
\[\mathbb{E}[\| e_n - e_m\|^{2}] \to 0 \quad \text{and} \quad  \mathbb{E}[\| e_m - e_k\|^{2}] \to 0 \]
as $ k \to \infty.$
Hence, the desired result follows directly from inequality~\eqref{tringular}.
\end{proof}
Building upon the preparatory results, particularly the characterization of the event~\(\daleth\) in~\eqref{event daleth}, we now proceed to establish the almost sure convergence of Algorithm~\ref{alg:Exact-data IRSGD}. This theorem represents the central outcome of the current subsection.

\begin{theorem}\label{thm for noisy}
  Let all the conditions of Lemma $\ref{first lemma}$ and  $\ref{condn for lemma 4}$ hold. Then, the sequence $\{u_k\}_{k\geq 1}$ produced by Algorithm~\ref{alg:Exact-data IRSGD} converges almost surely to an element of $\mathcal{D}_{c(\rho)}(\hat{u})$ in $U$, i.e.,
     \[ \lim_{k \to 0} \mathbb{E}[\|u_k - u^\dagger\|^2] = 0, \quad u^\dagger \in \mathcal{D}_{c(\rho)}(\hat{u}). \]
\end{theorem}
\begin{proof}
Lemma~\ref{uk convergent} establishes that the sequence of iterates \(\{u_k\}_{k \geq 1}\) produced by Algorithm~\ref{alg:Exact-data IRSGD} forms a Cauchy sequence in \(U\). Consequently, the sequence admits a limit, denoted by \(u^\dagger \in U\).
 Since the operator \(F\) is continuous, it follows that \(F(u_k) \to F(u^\dagger)\) as \(k \to \infty\).
Moreover, Corollary~\ref{for sum} guarantees that  
\[
\lim_{k \to \infty} \mathbb{E}\left[\|F(u_k) - y^\dagger\|^2 \right] = 0,
\]  
which implies that \(F(u^\dagger) = y^\dagger\). Given that each iterate \(u_k\) lies within the ball \(B_{c(\rho)}(\hat{u})\), we conclude that \(u^\dagger \in B_{c(\rho)}(\hat{u})\) almost surely. Consequently, \(u^\dagger\) belongs to the admissible set \(\mathcal{D}_{c(\rho)}(\hat{u})\) with probability one, thereby completing the proof.
\end{proof}


\subsection{Regularization Quality}\label{subsection noisy case} 
The primary goal of this subsection is to demonstrate that the IRSGD method possesses the regularizing property when employed with the proposed heuristic stopping rule. Our approach begins by establishing the pathwise stability of the IRSGD iterates \(u_k^\delta\) as \(\delta \to 0\), thereby linking Algorithm~\ref{alg:IRSGD Algorithm} for noisy data to its exact-data counterpart in Algorithm~\ref{alg:Exact-data IRSGD}. The proof structure is adapted from the methodology outlined in \cite[Lemma~3.6]{jin2020convergence}.

\begin{lemma} \label{first lemma for noisy}
   Assume that every condition in Lemma $\ref{first lemma}$ is met. For any given \( k \in \mathbb{N} \) and any realization of the sequence \((i_1, \ldots, i_{k-1}) \in \mathcal{F}_k\), let \( u_k \) and \( u_k^\delta \) be the iterates produced by the IRSGD method $\ref{SIRLI 2}$ for the exact and noisy data, respectively. Then
\[
\lim_{\delta \to 0^+} \mathbb{E} [\| u_k^\delta - u_k\|^2] = 0.
\]

\end{lemma}
\begin{proof}
We use induction to prove the assertion. For \( k = 0 \) it is trivial. Assume that it is true for any path in \( \mathcal{F}_k \) and for all indices up to \( k \). According to \ref{SIRLI 2}, for any fixed path \((i_1, \ldots, i_k)\), we get
\begin{align*}
    u_{k+1}^\delta - u_{k+1} &= (1 -\lambda_k)(u_k^\delta - u_k) - \omega_k ( (F_{i_k}'(u_k^\delta)^* - F_{i_k}'(u_k)^*)(F_{i_k}(u_k^\delta) - y_{i_k}^\delta)\\
    & \quad + F_{i_k}'(u_k)^* ((F_{i_k}(u_k^\delta) - y_{i_k}^\delta) - (F_{i_k}(u_k) - y_{i_k}^\dagger)) ).
\end{align*}
Thus, by the triangle inequality, we get
\begin{align} \label{first in lemma 5}
    \| u_{k+1}^\delta - u_{k+1} \| \nonumber &\leq (1-\lambda_k)\| u_k^\delta - u_k \| + \omega_k \| F_{i_k}'(u_k^\delta)^* - F_{i_k}'(u_k)^* \| \| F_{i_k}(u_k^\delta) - y_{i_k}^\delta \|\\
    &+ \omega_k \| F_{i_k}'(u_k)^* \| \| (F_{i_k}(u_k^\delta) - y_{i_k}^\delta) - (F_{i_k}(u_k) - y_{i_k}^\dagger) \|.
\end{align}
We proceed to establish the following statements for any fixed \(k\)

\begin{enumerate}
    \item[(i)] The supremum \( \sup_{(i_1, \ldots, i_{k-1}) \in \mathcal{F}_k} \| u_k \| \) is bounded. 
    \item[(ii)] The supremum of the difference \( \| u_k^\delta - u_k \| \) over all realizations in \(\mathcal{F}_k\), i.e., \( \sup_{(i_1, \ldots, i_{k-1}) \in \mathcal{F}_k} \| u_k^\delta - u_k \| \), is uniformly bounded.
\end{enumerate}
To prove claim~(i), we invoke \eqref{bddness of F_i'} and \eqref{tangential}, which yields that
\begin{align*}
    \| u_{k+1} - \hat{u} \| &\leq \| u_k - \hat{u} \| + \omega_k \| F_{i_k}'(u_k)^* \| \| F_{i_k}(u_k) - y_{i_k}^\dagger \| + \lambda_k \| u_0 - u_k  \| \\
    &\leq \left( 1 + \lambda_k + \omega_k \frac{L^2}{1 - \eta} \right) \| u_k - \hat{u} \| + \lambda_k \|u_0 - \hat{u}\|.
\end{align*}
This recursive inequality implies that the sequence \(\{ \| u_k - \hat{u} \| \}\) grows at most linearly, and by applying an induction argument, we conclude that  claim~(i) is established.
Similarly, to establish claim (ii), we note that
\begin{align}\label{3.36}
  \nonumber  \| F_{i_k}(u_k^\delta) - y_{i_k}^\delta \| &\leq \| F_{i_k}(u_k^\delta) - F_{i_k}(u_k) \| + \| F_{i_k}(u_k) - y_{i_k}^\dagger \| + \| y_{i_k}^\dagger - y_{i_k}^\delta \| \\
    &\leq \frac{L}{1 - \eta} \left( \| u_k^\delta - u_k \| + \| u_k - \hat{u} \| \right) + \|y^\delta - y^\dagger\|,
\end{align}
and consequently \ref{first in lemma 5} and \ref{3.36} derive that
\begin{align*}
   \| u_{k+1}^\delta - u_{k+1} \| &\leq  \omega_k \left( \frac{L}{1 - \eta} \left( \| u_k^\delta - u_k \| + \| u_k - \hat{u} \| \right) + \|y^\delta - y^\dagger\| \right) \| F_{i_k}'(u_k^\delta)^* - F_{i_k}'(u_k)^* \|  \\
   &+ \omega_k L \| (F_{i_k}(u_k^\delta) - y_{i_k}^\delta) - (F_{i_k}(u_k) - y_{i_k}^\dagger) \| + (1 - \lambda_k)\| u_k^\delta - u_k \|. 
\end{align*}
This, together with a standard induction argument, confirms claim~(ii). Consequently, leveraging both claims, we may now proceed to let
\[ c = \frac{L}{1 - \eta} \sup_{(i_1, \ldots, i_{k-1}) \in \mathcal{F}^k} (\| u_k^\delta - u_k \| + \| u_k - \hat{u} \|) + \|y^\delta - y^\dagger\|. \]
Then, it follows from \ref{first in lemma 5}, \ref{3.36} and \ref{bddness of F_i'} that
\begin{align*}
    \lim_{\delta \to 0^+} \| u_{k+1}^\delta - u_{k+1} \| &\leq (1-\lambda_k)\lim_{\delta \to 0^+} \| u_k^\delta - u_k \| + c \omega_k \lim_{\delta \to 0^+} \| F_{i_k}'(u_k^\delta)^* - F_{i_k}'(u_k)^* \| \\
    &+ \omega_k L \lim_{\delta \to 0^+} \| (F_{i_k}(u_k^\delta) - y_{i_k}^\delta) - (F_{i_k}(u_k) - y_{i_k}^\dagger) \|.
\end{align*}
Finally, by the continuity of the operators \(F_i\) and \(F_i'\), we obtain  
\(
\lim_{\delta \to 0^{+}} \|u_{k}^\delta - u_{k}\| = 0,
\)
which, in turn, implies  
\(
\lim_{\delta \to 0^{+}} \mathbb{E}\big[ \|u_{k}^\delta - u_{k}\|^2 \big] = 0.
\)
This completes the proof.
\end{proof}
We are now prepared to formally demonstrate the central result of this subsection, which concerns the regularization property of the IRSGD method.

\begin{theorem}\label{main result}
  Assume that every condition in Lemma $\ref{first lemma}$ is met. Suppose that $\{y^\delta\}$ is the noisy data sequence such that Assumption $\ref{new assumption}$ hold and let IRSGD iteration corresponding to noisy data is stopped at $k_{*}(y^\delta)$ defined via Rule 1. Then there exist a solution $\hat{u}$ of $\ref{stcombined}$ such that 
   \[\mathbb{E}[\|u_{k_{*}(y^\delta)}^\delta - \hat{u}\|^2] \to 0 \quad \text{ as} \quad \delta \to 0.\]
   \end{theorem}
\begin{proof}
 Let $\hat{u}$ be the limit of $u_k$, which solves \ref{stcombined} in $B_{\rho + c(\rho)}(u_0)$ along any sample path belonging to the event 
 $\daleth,$ where $\daleth$ is defined in \eqref{event daleth} with $\mathbb{P}(\daleth)=1.$ To establish the result, it suffices to show that
 \[
 \lim_{\delta \to 0 } \|u_{k_{*}(y^\delta)}^\delta - \hat{u}\|^2] = 0 \quad \text{on the event} \daleth.
 \]
Fix an arbitrary sample path in $\daleth$ along which we define  $K:= \lim \inf_{m \to \infty} k_{*}(y^{\delta_m}),$ where $\{y^{\delta_m}\}$ is a subsequence of $\{y^\delta\}.$ By taking a subsequence of $\{y^{\delta_m}\}$ if necessary, we may assume that $K = \lim_{m \to \infty}k_{*}(y^{\delta_m})$ and, according to Lemma \ref{first lemma for noisy}, we have
    \begin{equation}\label{I eqn in last thm}
        \mathbb{E}[\| u_{k}^{\delta_m} - u_k\|^2] \to 0 \quad \text{as} \quad m \to \infty\ \ \text{for all}\ 0 \leq k \leq K.
    \end{equation}
   There are two possible cases to consider:
\\
\textbf{Case I:} Suppose \( K < \infty \). Then, for sufficiently large $m$, we have $k_{*}(y^{\delta_m}) = K$. Thus, from Lemma \ref{secound lemma} it follows that 
    \[ \mathbb{E} \left[ \|F(u_{K}^{\delta_m}) - y^{\delta_m}\| \right] \rightarrow 0 \quad \text{as} \quad m \rightarrow \infty. \]
     By engaging \ref{I eqn in last thm} and the continuity of $F$, we have $F(u_K) = y^\dagger.$ In other words, $K$th iteration of IRSGD method with exact data is a solution of $F(u) = y^{\dagger}.$ Hence, the method terminates with $\hat{u} = u_K$ and $u_{K}^{\delta_m} \to \hat{u}$ as $m \to \infty.$

 \textbf{Case II:} Suppose $K = \infty$. Then  Theorem \ref{thm for noisy} and Lemma \ref{first lemma for noisy} guarantees that, for each $\epsilon >0,$ there exist  $\hat{k}\in \mathbb{N}$ and $\overline{m}\in \mathbb{N}$ such that 
     \begin{equation}\label{3.38}
          \mathbb{E}[\|u_{\hat{k}} - \hat{u}\|^2] < \frac{\epsilon}{8}, \quad \text{and} \quad \mathbb{E}[\|u_{\hat{k}}^{\delta_m} - u_{\hat{k}}\|^2] < \frac{\epsilon}{8} \quad \text{for all} \quad m \geq \overline{m}.          
     \end{equation}
     At this stage, we  apply the full conditional on \ref{eqn_of_lemma1_part2} to arrive at
     \begin{equation}\label{3.39}
          \mathbb{E}[\|u_{K}^{\delta_m} - \hat{u}\|^2] \leq 2\mathbb{E}[\|u_{\hat{k}}^{\delta_m} - u_{\hat{k}}\|^2 + 2\mathbb{E}[\|u_{\hat{k}} - \hat{u}\|^2 + C_2 \sum_{l = \hat{k}}^{K-1} \lambda_l +  C_3\sum_{l = \hat{k}}^{K-1}\frac{1}{(k + a)^2}  ,
     \end{equation}
     where $C_2 = \rho^2 (1 + \kappa^{-2}\Omega^2 L^2) + 2\rho c(\rho)$ and $C_3 = \frac{M\Omega(1+ \eta)}{\tau}.$ By choosing $\hat{k}$ such that
     \[\sum_{l = \hat{k}}^{K-1} \lambda_l \leq \frac{\epsilon}{4C_2} \quad \text{and} \quad \sum_{l = \hat{k}}^{K-1}\frac{1}{(k + a)^2} \leq \frac{\epsilon}{4C_3},\] and plugging these estimates along with \ref{3.38} in \ref{3.39}, we finally get
     \[\mathbb{E}[\|u_{K}^{\delta_m} - \hat{u}\|^2] \leq \epsilon \quad \text{for all} \quad m \geq \overline{m}. \]
     The above reasoning establishes that $\|u_{k_{*}(y^\delta)}^\delta - \hat{u}\| \to 0$ along any sample path in $\daleth$. Consequently we have  $\|u_{k_{*}(y^\delta)}^\delta - \hat{u}\| \to 0$ as $\delta \to 0$, almost surely.
   Thus, the proof is finished.
\end{proof}
\section{Numerical Validations and Discussions}\label{Section:NE}
In this section, we present a series of numerical experiments to assess the practical performance of the proposed scheme, implemented in conjunction with Rule~\ref{Hanke Rule}. The experiments are designed to illustrate the effectiveness of the method in solving ill-posed inverse problems, encompassing both linear and nonlinear cases.

\subsection{Linear case}
We consider the following lnear system.
 \begin{equation}\label{numsystem}
     F_{i}(u) =  \int_{a}^{b} k(t_{i},s)u(s) \,ds = y_{i}, \hspace{5mm} i = 0, 1, \ldots, P-1,
 \end{equation}
which is derived from a linear integral equation on $[a,b]$, discretized by sampling at points $t_{i} \in [a,b]$ for $i = 0, 1, \ldots, P-1,$ where $t_{i} = a + \frac{i(b-a)}{P-1}.$ For a comprehensive discussion on such problems, see \cite{jin2023convergence}. It is assumed that the kernel function \( k(\cdot,\cdot) \) maintains continuity throughout the domain \( [a,b] \times [a,b] \).
 To evaluate the effectiveness of  the IRSGD method \ref{SIRLI 2}, we apply it to system \ref{numsystem} with $P=1000, [a, b] = [-6,6],$ and the kernel function given by $k(t,s)= \varphi(t-s),$ where
 \(\varphi(t) = (1+\cos(\frac{\pi t}{3}))_{\chi_{\{|t|<3\}}}.\)
The goal is to reconstruct the target solution given as  
\[
u^{\dagger}(s) = \sin\left(\frac{\pi s}{12}\right) + \sin\left(\frac{\pi s}{3}\right) + \frac{1}{200} s^2 (1 - s).
\]
In place of $y_i= F_i(u^{\dagger})$ (exact data), we consider noisy observations given by
\begin{equation}
    y^{\delta}_i = y_i +  {\epsilon_i}\delta_{\text{rel}}|y_i|, \quad i = 0, 1, \ldots, P-1,
\end{equation}
where \(\delta_{\text{rel}}\) denotes the relative noise level and $\epsilon_i$ represents random noise drawn from a standard Gaussian distribution  $N(0,1)$. 
To discretize the problem, we partition the interval $[-6,6]$ into $P-1$ equal subintervals and approximate the integrals using the trapezoidal rule. This leads to a discretized linear system of dimension $n = 1000$ with the corresponding forward operator $F_{n\times n}$. We execute the method \ref{SIRLI 2} with the numerical settings as follows:
\begin{itemize}
    \item [(a)] Initial guess $u_0 = 0.$
     \item[(b)] The step size is chosen as $\omega_k = \omega = 10^{-3}.$
    \item[(c)]  The sequence of weighted parameters $\{\lambda_k\}_{k\geq 1}$ is chosen with $\lambda_k = \frac{1}{k^2}.$
    \item[(d)] The maximum number of iterations, denoted by $k_{\text{max}}$, is set to $1000$.   
    \item [(e)] Noisy data is employed with four noise levels given by $\delta_{\text{rel}} = 10^{-1}, 10^{-2}, 10^{-3}$ and $10^{-4}$.
\end{itemize}

We mainly perform the three experiments as:

 \subsubsection{Convergence behavior of $\Psi (k, y^\delta)$}
 In order to investigate the convergence behavior of $\Psi (k, y^\delta)$ as $\delta \to 0$, we present the plots of the values of $\Psi(k, y^\delta)$ versus $k$ in Figure \ref{fig:Psi_convergence} using the fixed constant $a=100$ (as in Rule \ref{Hanke Rule}) and for different levels of noise. 
\begin{figure*}[htbp]
    \centering
    \includegraphics[width=1.05\textwidth,keepaspectratio]{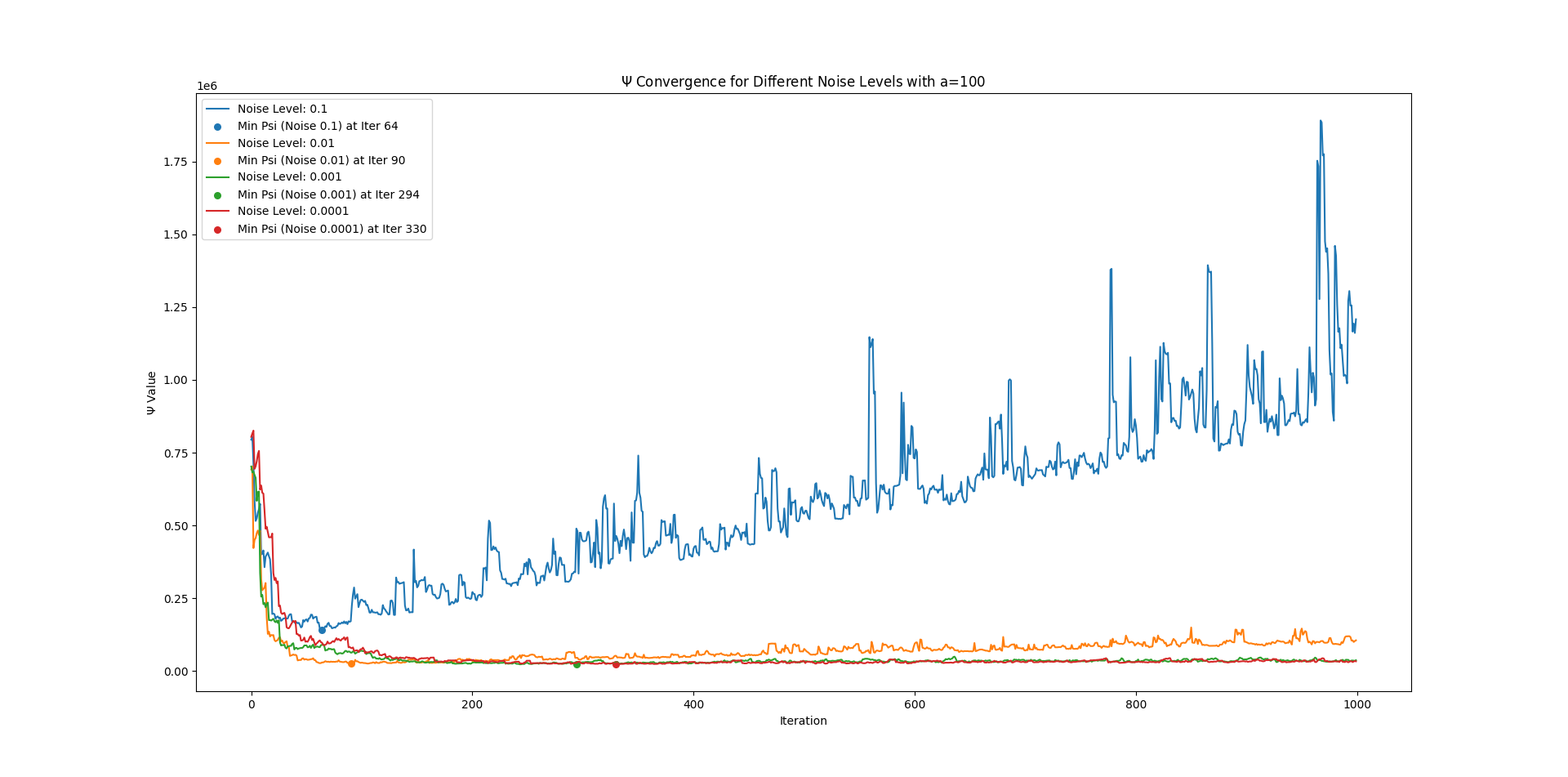}
    \caption{Convergence behavior of $\Psi(k, y^\delta)$ for varying noise levels, illustrating the sensitivity of the proposed stopping rule to measurement perturbations.}
    \label{fig:Psi_convergence}
\end{figure*}

Figure \ref{fig:Psi_convergence} clearly shows that the noise level in the input data has a significant effect on the convergence behavior of $\Psi (k, y^\delta)$. As $\delta_{\text{rel}}$ decreases from $10^{-1}$ to $10^{-4}$, convergence becomes both faster and more stable. Additionally, it demonstrates that $\Psi (k_*, y^\delta) \to 0$ as $\delta \to 0$, which supports the assertion in Lemma \ref{secound lemma}. Furthermore,  as $\delta$ decreases, the stopping index $k_*$ increases.

 \subsubsection{Reconstructed solutions}
We reconstructed the solution using \ref{SIRLI 2} at various noise levels, as shown in Figures \ref{M=0.1}-\ref{M=0.001}. These results illustrate the impact of the constant \(a\) (as defined in Rule \ref{Hanke Rule}) on reconstruction accuracy.
 At a noise level of $10^{-1}$, smaller values of $a$ lead to early stopping of the algorithm, resulting in an inaccurate reconstruction. As $a$ increases (e.g., $a = 100$ and $a = 1000$), the reconstruction accuracy improves significantly.
 \begin{figure}[htbp]
    \centering
    \subfigure[]{\includegraphics[width=0.48\textwidth]{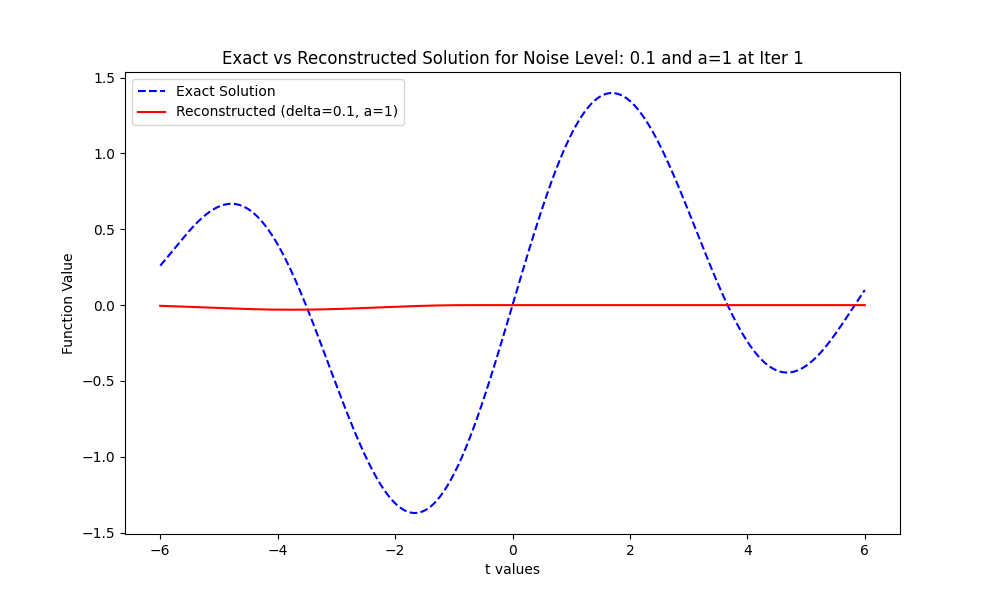}} 
     \subfigure[]{\includegraphics[width=0.48\textwidth]{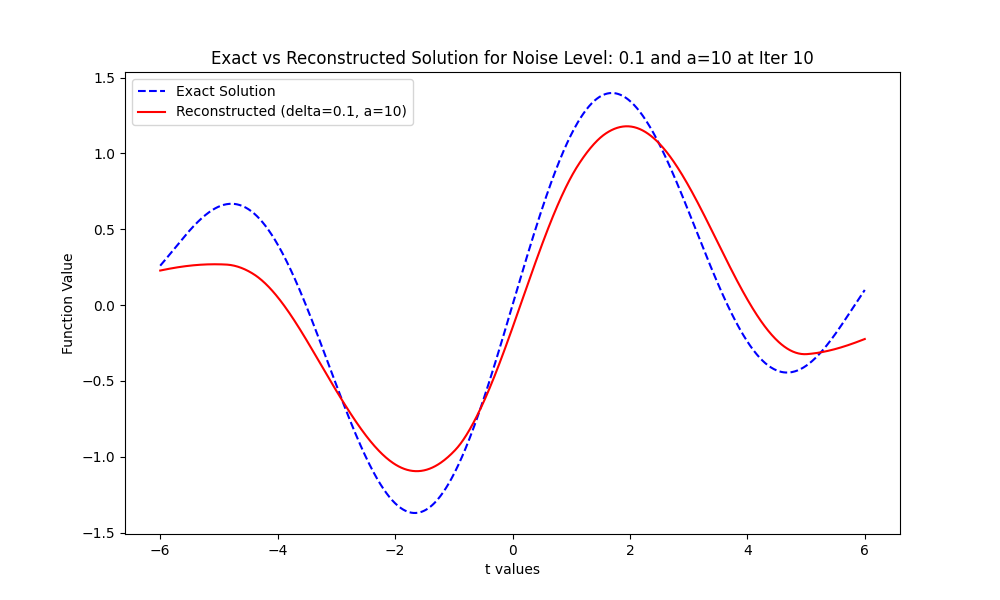}}
 \subfigure[]{\includegraphics[width=0.48\textwidth]{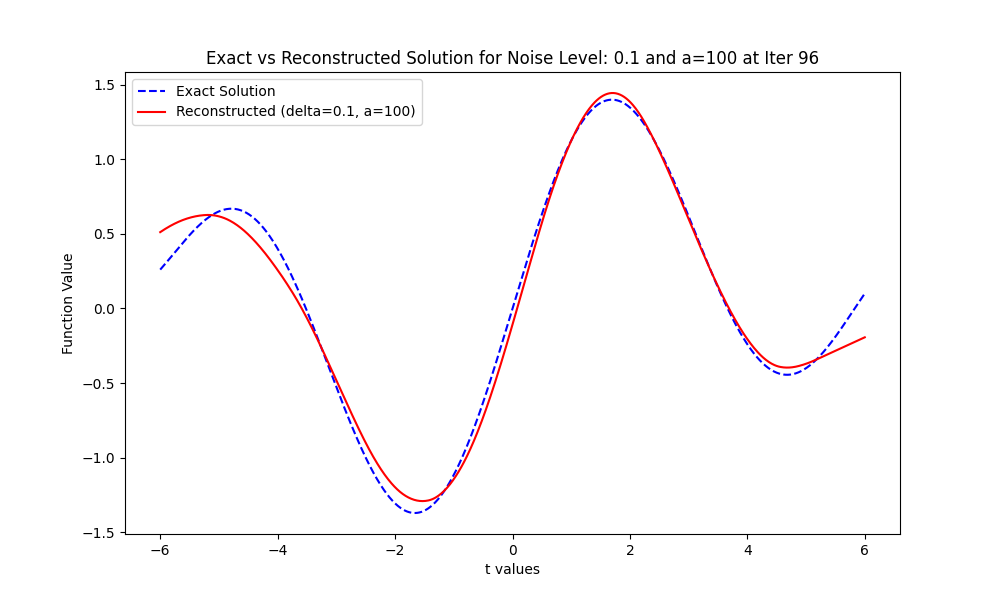}}
  \subfigure[]{\includegraphics[width=0.48\textwidth]{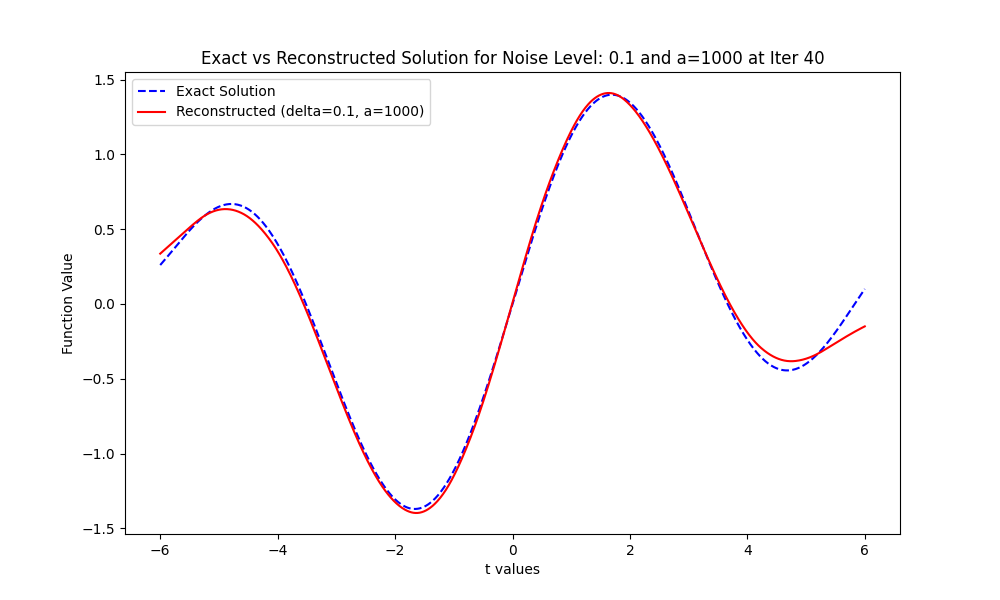}}
      \caption{Effect of the parameter $a$ on reconstruction quality for $\delta_{\text{rel}} = 10^{-1}$: (i) $a = 1$, (ii) $a = 10$, (iii) $a = 100$, and (iv) $a = 1000$, showing side-by-side comparisons of the exact and reconstructed solutions.}
    \label{M=0.1}
\end{figure}
From Figures \ref{M=0.01} and \ref{M=0.001} observe that, as the noise level decreases, the inaccuracy decreases for small values of $a$ (e.g., $a = 1$ and $a = 10$). However, for $a = 1000$, the number of iterations increases significantly compared to the other cases, leading to a higher computational burden. Therefore, for lower noise levels, smaller values of $a$ should be chosen to reduce computational cost and achieve faster results.

\begin{figure}[htbp]
    \centering
    \subfigure[]{\includegraphics[width=0.48\textwidth]{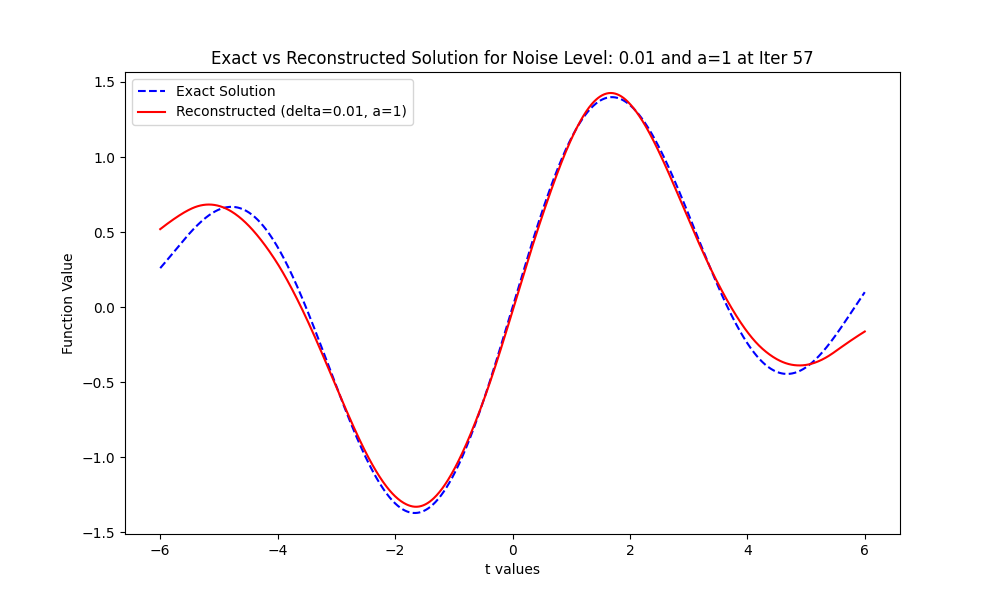}}
    \subfigure[]{\includegraphics[width=0.48\textwidth]{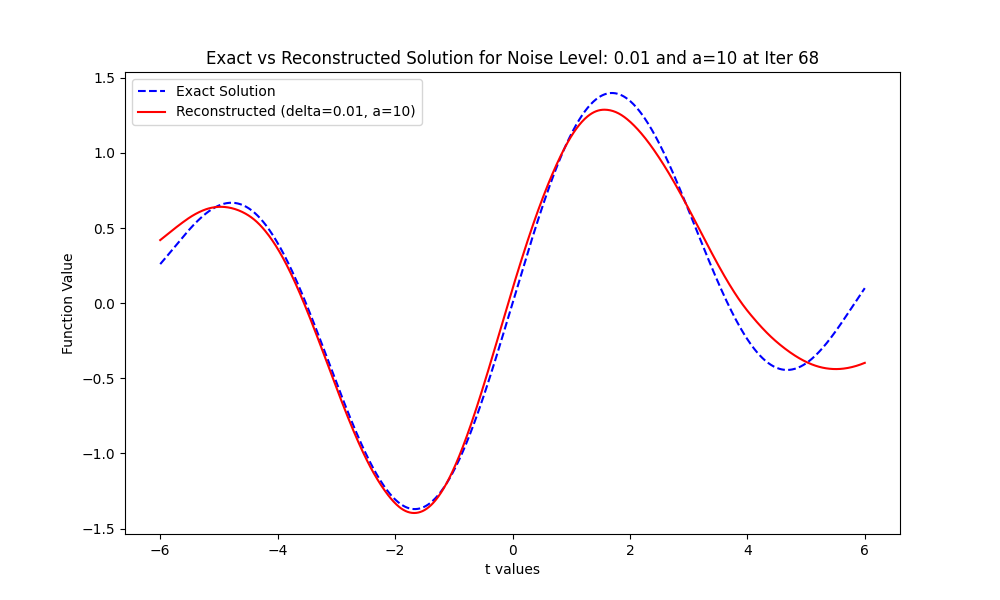}}
    \subfigure[]{\includegraphics[width=0.48\textwidth]{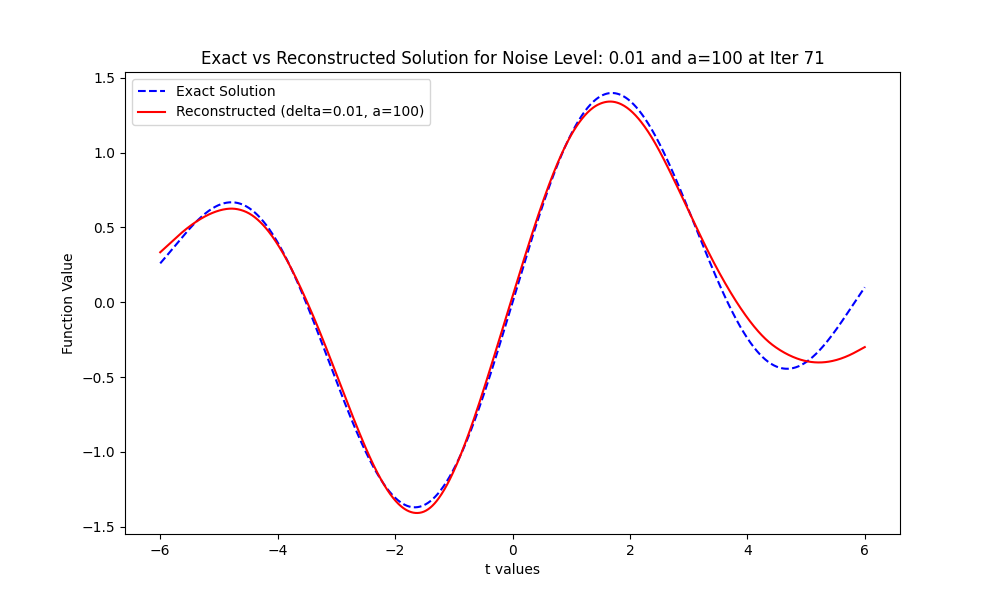}}
    \subfigure[]{\includegraphics[width=0.48\textwidth]{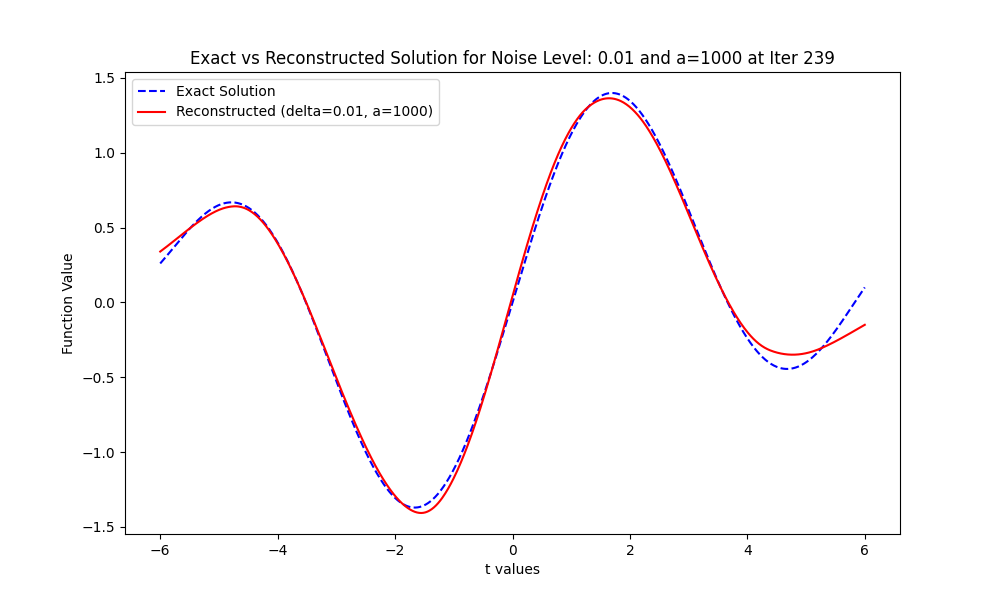}}
    
      \caption{Effect of the parameter $a$ on reconstruction quality for $\delta_{\text{rel}} = 10^{-2}$: (i) $a = 1$, (ii) $a = 10$, (iii) $a = 100$, and (iv) $a = 1000$, showing side-by-side comparisons of the exact and reconstructed solutions.}
    \label{M=0.01}
\end{figure}
\begin{figure}[htbp]
    \centering
    \subfigure[]{\includegraphics[width=0.48\textwidth]{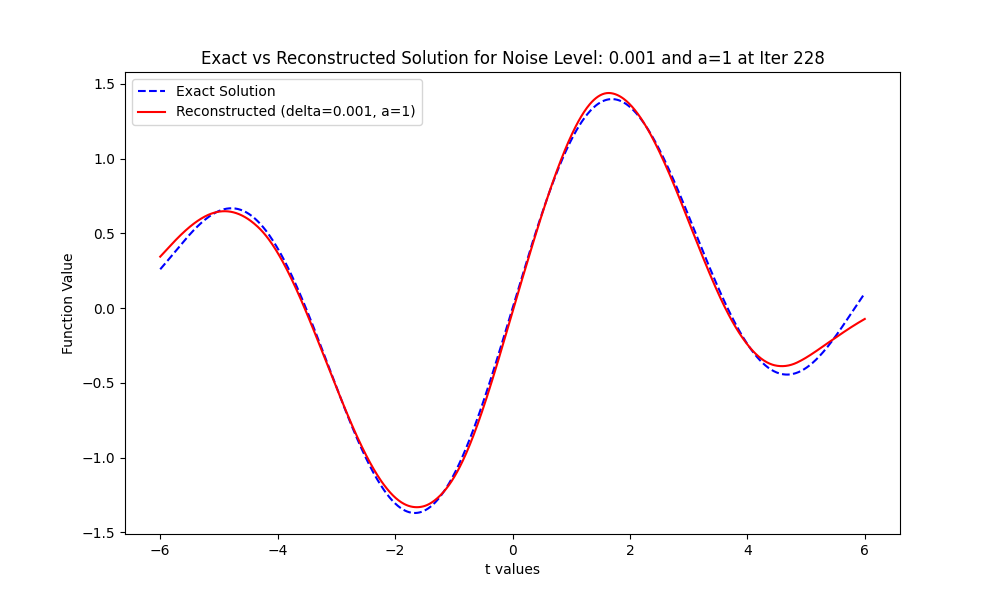}} 
     \subfigure[]{\includegraphics[width=0.48\textwidth]{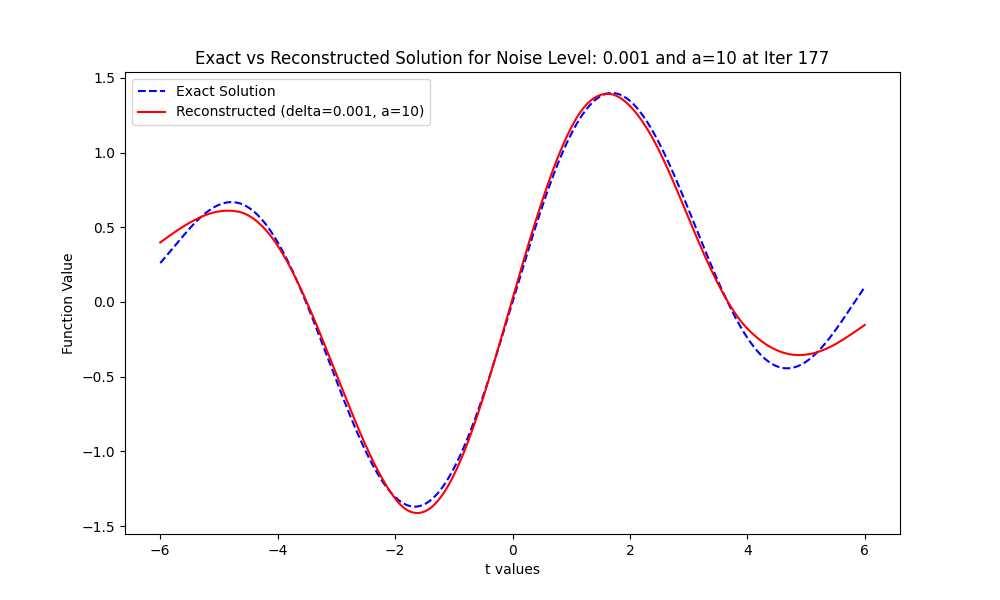}}
 \subfigure[]{\includegraphics[width=0.48\textwidth]{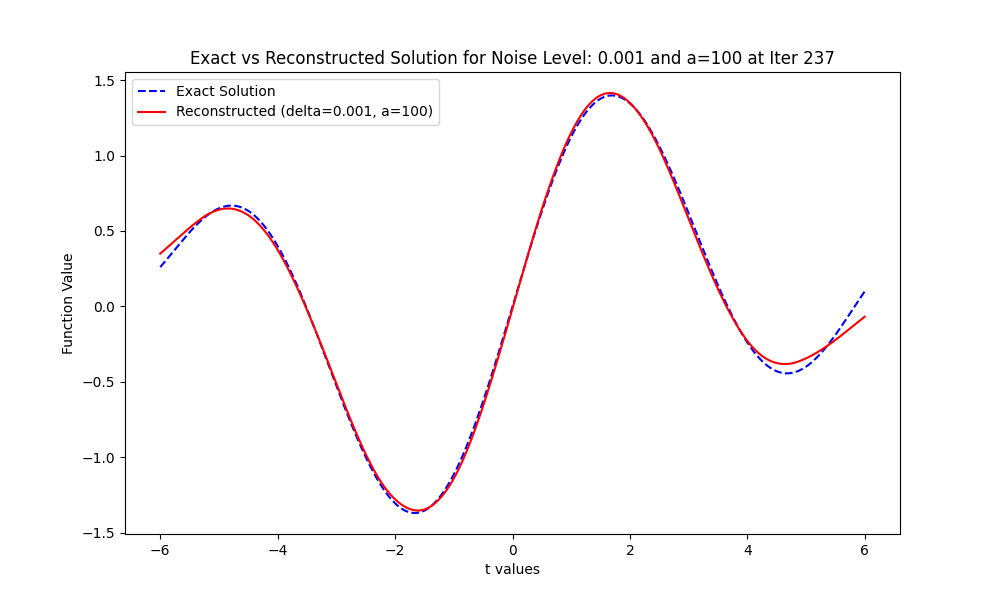}}
  \subfigure[]{\includegraphics[width=0.48\textwidth]{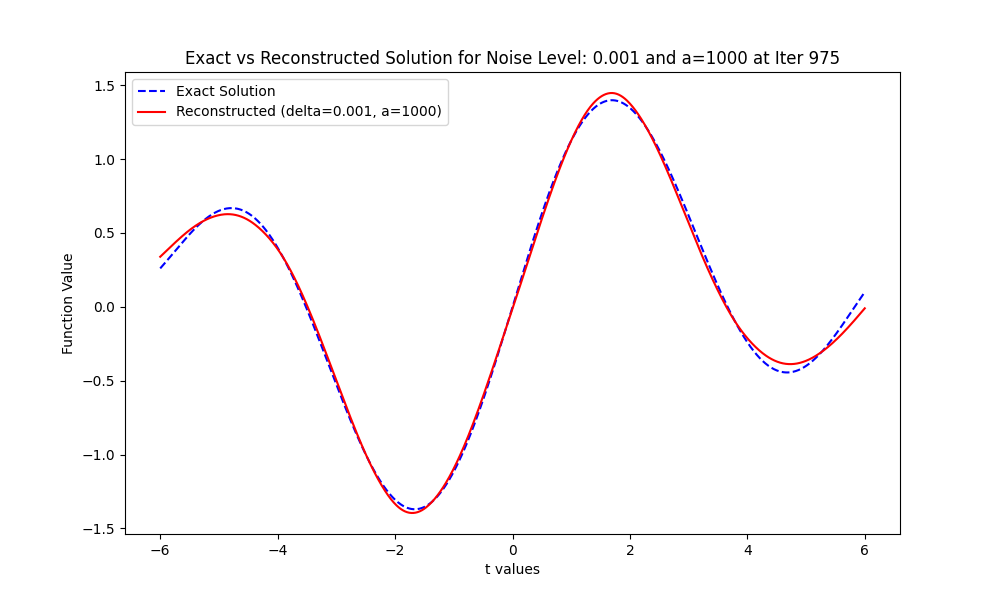}}
   \caption{Effect of the parameter $a$ on reconstruction quality for $\delta_{\text{rel}} = 10^{-3}$: (i) $a = 1$, (ii) $a = 10$, (iii) $a = 100$, and (iv) $a = 1000$, showing side-by-side comparisons of the exact and reconstructed solutions.}
    \label{M=0.001}
\end{figure}

Tables \ref{IRSGD with different a values} provide detailed numerical results (stopping index $k_*$ and the relative errors $E^\delta_k := \frac{\| u^\delta_k - u^\dagger \|_{L^2(\Omega)}}{\| u^\dagger \|_{L^2(\Omega)}}$) for Rule \ref{Hanke Rule} at various noise levels.  
\begin{table}[ht]
\centering
\begin{tabular}{ccccccccccc}
\toprule
 $\delta_{\text{rel}}$ & &  & $a$ in Rule \ref{Hanke Rule} & & \\
\midrule
 & \multicolumn{2}{c}{$1\text{e}1$} & \multicolumn{2}{c}{$1\text{e}2$} & \multicolumn{2}{c}{$1\text{e}3$} & \multicolumn{2}{c}{$1\text{e}4$} \\
\cmidrule(r){2-3} \cmidrule(r){4-5} \cmidrule(r){6-7} \cmidrule(r){8-9} 
 & $k_*$ & $E_{k_*}^\delta$ & $k_*$ & $E_{k_*}^\delta$ & $k_*$ & $E_{k_*}^\delta$ & $k_*$ & $E_{k_*}^\delta$  \\
\midrule
$1\text{e}^{-1}$ & 18 & $3.550\text{e}^{-2}$ & 32 & $1.565\text{e}^{-2}$ & 49 & $1.441\text{e}^{-2}$ & 249& $4.213\text{e}^{-3}$  \\

$1\text{e}^{-2}$ & 133 & $4.817\text{e}^{-3}$ & 80 & $8.626\text{e}^{-3}$ & 506 & $1.300\text{e}^{-3}$ & 565 & $1.664\text{e}^{-3}$   \\

$1\text{e}^{-3}$ & 172 & $2.293\text{e}^{-3}$ & 260 & $4.982\text{e}^{-3}$ & 2044 & $5.845\text{e}^{-4}$ & 3404 & $4.982\text{e}^{-4}$ \\
$1\text{e}^{-4}$ & 3304 & $5.341\text{e}^{-4}$ & 3167 & $5.134\text{e}^{-4}$ & 867 & $6.111\text{e}^{-4}$ & 4874 & $3.185\text{e}^{-4}$   \\
\bottomrule
\end{tabular}
\caption{Performance of the IRSGD method for various values of the parameter $a$ in Rule~\ref{Hanke Rule} and different relative noise levels $\delta_{\text{rel}}$. For each combination, the table reports the stopping index $k_*$ and the corresponding reconstruction error $E_{k_*}^\delta$.}
\label{IRSGD with different a values}
\end{table}
\subsubsection{Comparison between SGD and  IRSGD}  For this comparison, we set the constant \( a \) in the rule (referred to as Rule \ref{Hanke Rule}) to \( a = 100 \). 
The numerical findings  are reported in Table \ref{comparison-table}.
\begin{table}[ht]
\centering
\begin{tabular}{lcccc}
\toprule
\textbf{Methods} & \textbf{$\delta_{\text{rel}}$} &  \textbf{$k_*$} & \textbf{$E_{k_*}^{\delta}$} & \textbf{$\Psi(k_*, y^\delta)$} \\
\midrule
IRSGD & $1e^{-1}$ & 68 & $1.58e^{-2}$ & 139237.06 \\
SGD & $1e^{-1}$ & 66 & $3.34e^{-2}$  & 153135.32 \\
\midrule
IRSGD &  $1e^{-2}$ & 180 & $2.49e^{-3}$ & 24514.94 \\
SGD & $1e^{-2}$ & 192 & $1.44e^{-3}$ &27692.34\\
\midrule
IRSGD & $1e^{-3}$ &  258 & $5.53e^{-3}$ & 20785.89\\
SGD &  $1e^{-3}$ & 772 & $3.40e^{-4}$ & 8772.18\\
\midrule
IRSGD & $1e^{-4}$ & 3492 & $7.93e^{-4}$ & 21809.47 \\
SGD & $1e^{-4}$ & 4122 & $4.59e^{-4}$ & 12432.96 \\
\midrule
IRSGD & $1e^{-5}$ & 3468 & $6.24e^{-4}$ & 17168.44 \\
SGD & $1e^{-5}$ & 3529 & $7.71e^{-4}$ & 17952.29 \\
\bottomrule
\end{tabular}
\caption{Quantitative comparison between IRSGD and SGD methods using Rule~\ref{Hanke Rule} for various relative noise levels $\delta_{\text{rel}}$. The table reports the stopping index $k_*$, reconstruction error $E_{k_*}^\delta$, and $\Psi(k_*, y^\delta)$ for each case.}
\label{comparison-table}
\end{table}

We observe that, as noise level decreases (moving down the table), both methods generally require more iterations (\( k_* \)) to achieve an acceptable solution. Lower noise levels yield lower relative errors \( E_{k_*}^{\delta} \), which is expected as there is less perturbation in the data, allowing both methods to approximate the true solution more accurately.
For different noise levels, IRSGD tends to converge in fewer iterations than SGD, indicating that the additional damping term in IRSGD can speed up convergence. However, for very low noise (\( \delta_{\text{rel}} = 1 \times 10^{-5} \)), the iteration counts for IRSGD and SGD are similar, suggesting that both methods stabilize in a similar time frame.

\vspace{-0.2cm}

\subsection{Nonlinear case}
This subsection presents the mathematical model that illustrates the Schlieren tomography and validates the proposed methods. The Schlieren optical system produces a light intensity pattern that is proportional to the square of the line integral of pressure along the light path as it passes through the tank (see \cite{Schileren2007} for more details).

For each \( i \) between \( 0 \) and \( P - 1 \), let \( \sigma_i \in S^1 \) be a set of recording directions such that \( \zeta_i = \zeta(\theta_i) = (\cos \theta_i, \sin \theta_i) \). The objective is to recover a function \( u \) supported within the bounded domain \( D := [-1, 1] \times [-1, 1] \subset \mathbb{R}^2 \) from the corresponding system of equations
\[
F_i(u) = y_i^\delta, \quad \text{for } i = 0, 1,  \ldots, P - 1,
\]
where the Schlieren transform \( F_i = R_i^2 \), with the operator \( R_i \) denoting the Radon transform along the direction \( \zeta_i \), satisfying \( R_i : C_0^\infty(D) \to C_0^\infty(I) \):
\[
u \mapsto R_i(u) := \left( s \mapsto \int u(s \sigma_i + r \sigma_i^\perp) \, dr \right).
\]
From \cite{Schileren2007}, we know that each operator \( F_i : H_0^1(D) \to L^2(I) \) with \( I = [-1, 1] \) is continuous and Fr\'echet differentiable with
\[
F_i'(u)h = 2R_i(u)R_i(h), 
\]
for all \(h \in H_0^1(D).\) Moreover, the adjoint \( F_i'(u)^* : L^2(I) \to H_1^0(D) \) is given by
\[
F_i'(u)^* g = (I - \Delta)^{-1} \left( 2R_i^*(R_i(u) g) \right), \quad \forall g \in L^2(I),
\]

where \( R_i^* : L^2(I) \to L^2(D) \) denotes the adjoint of  \( R_i \), defined by \( (R_i^* g)(u) = g(u \cdot \sigma_i),\) $I$ refers to the identity operator and $\Delta$  represents the Laplace operator acting on $H_{0}^{1}(D)$. Evaluating the adjoint operator entails the solution of an elliptic boundary value problem, for which a finite difference discretization is employed within a multigrid framework to obtain an efficient numerical approximation.

In this numerical experiment, the parameters were set as follows: 
$\delta_{\text{rel}} = 10^{-2}$, $k_{\text{max}} = 1000$, $\omega_k = \omega = 5 \times 10^{-3}$, and $\lambda_k = k^{-3}$. 
The true image and the corresponding initial guess are presented in Figure~\ref{initial and true figures}. 
The proposed method was applied to reconstruct the image, with the stopping criterion determined by Rule~\ref{Hanke Rule} for various values of the parameter $a$, as illustrated in Figure~\ref{Reconstructed}.  

To assess the reconstruction performance, we monitored both the relative error and the evolution of $\Psi(k, y^\delta)$ for selected values of $a$, specifically $a = 10^2, a = 10^{3}$ and $a = 10^{4}$. 
The corresponding results are displayed in Figures~\ref{RE=1000}, \ref{RE= 10000} and \ref{RE= 100000}. 
Furthermore, the quality of the reconstruction was evaluated using the peak signal-to-noise ratio (PSNR) and the structural similarity index (SSIM) \cite{Wang2004}, with the obtained values reported in the Table~\ref{nonlinear-table}, alongside the stopping index and relative error for each case.

\begin{figure}[htbp]
    \centering
    \subfigure[]{\includegraphics[width=0.48\textwidth]{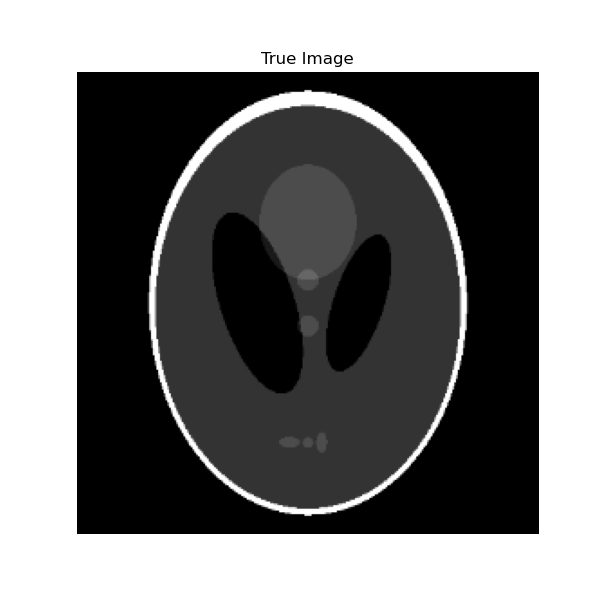}}
    \subfigure[]{\includegraphics[width=0.48\textwidth]{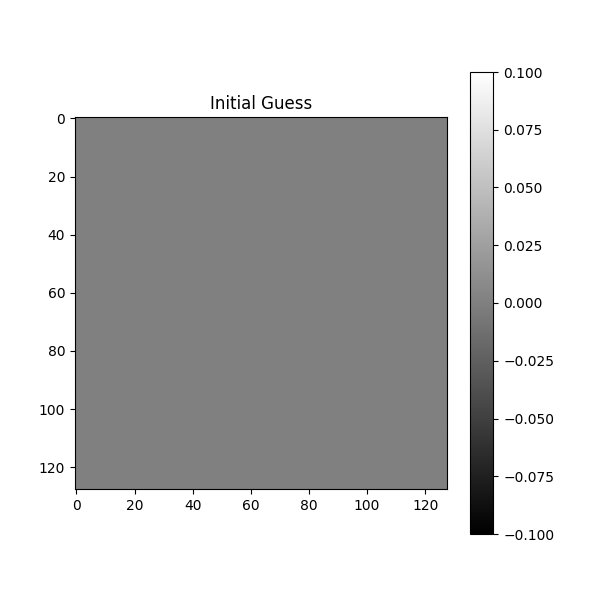}}
\caption{(i) True image of the Shepp–Logan phantom; (ii) corresponding initial guess for the reconstruction.}
\label{initial and true figures}
\end{figure}
\begin{figure}
    \centering
    \subfigure[]{\includegraphics[width=0.32\textwidth]{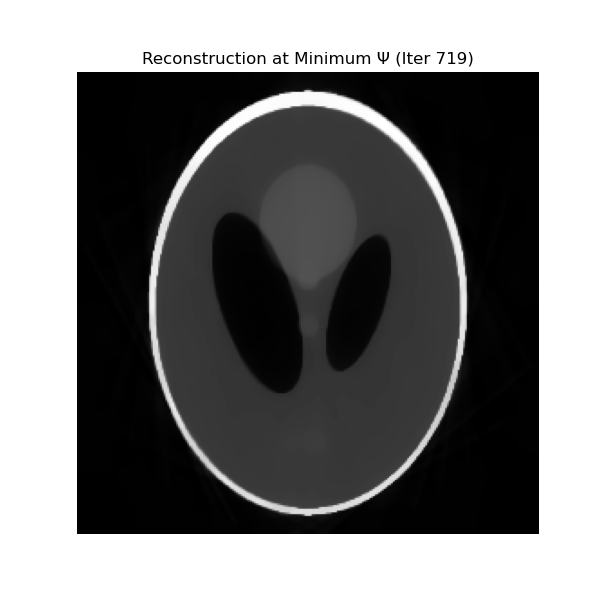}}
    \subfigure[]{\includegraphics[width=0.32\textwidth]{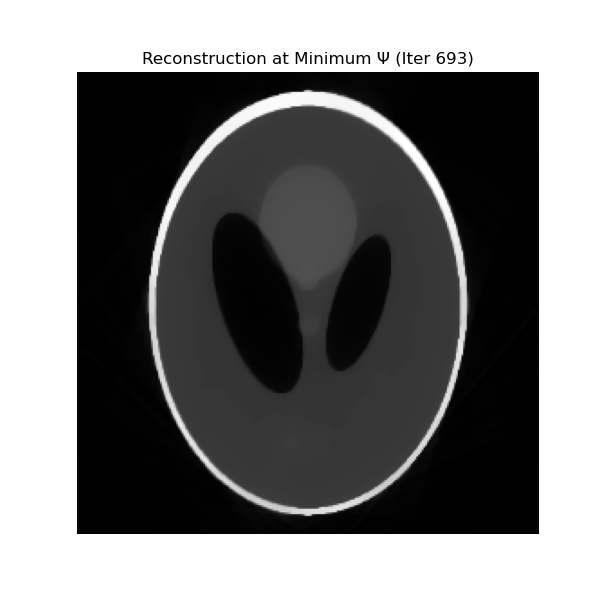}}
    \subfigure[]{\includegraphics[width=0.32\textwidth]{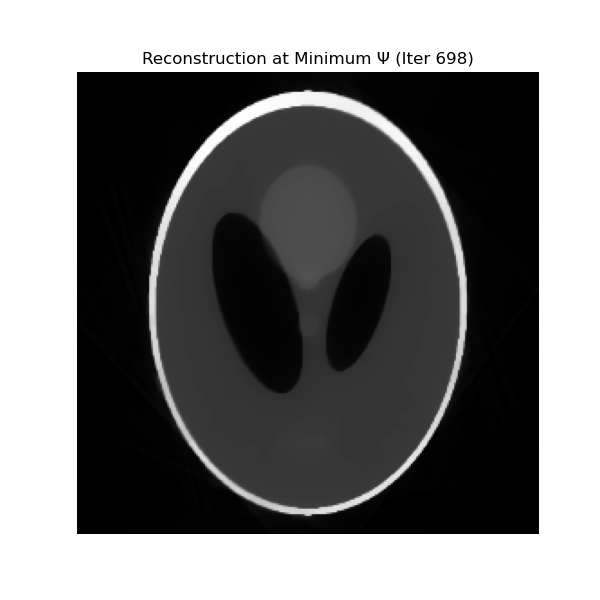}}  
    \caption{Reconstructed solutions for $\delta_{\text{rel}} = 10^{-2}$ using different values of the parameter $a$ in Rule~\ref{Hanke Rule}: (i) $a = 100$; (ii) $a = 1000$; (iii) $a = 10000$.}

    \label{Reconstructed}
\end{figure}

\begin{figure}
    \centering
    \subfigure[]{\includegraphics[width=0.48\textwidth]{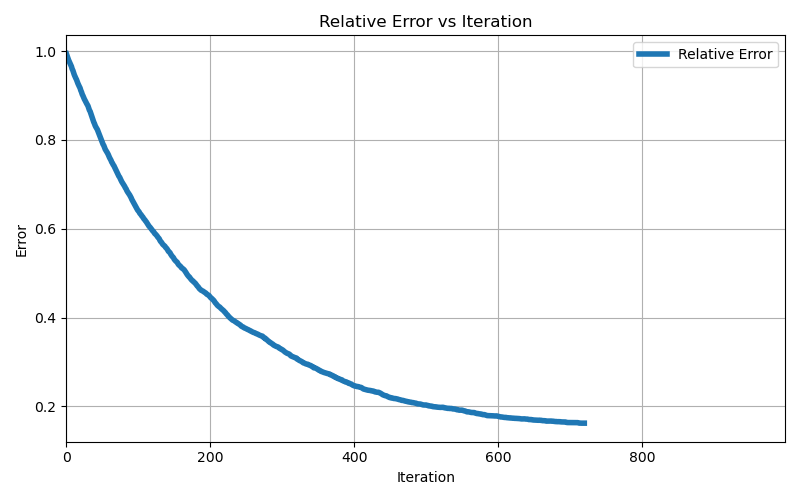}}
    \subfigure[]{\includegraphics[width=0.48\textwidth]{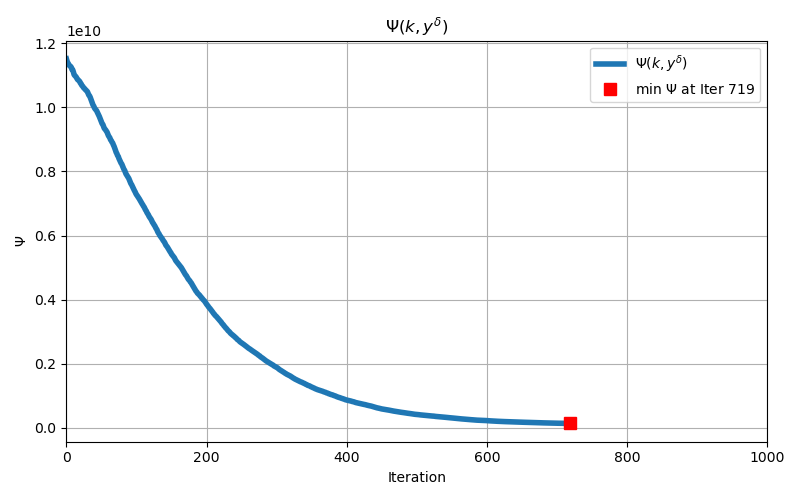}}  
    \caption{Results for $\delta_{\text{rel}} = 10^{-2}$ with $a = 100$: (i) relative error in the reconstructed solution; (ii) progression of $\Psi(k, y^\delta)$ over the iterations.}
    \label{RE=1000}
\end{figure}

\begin{figure}
    \centering
    \subfigure[]{\includegraphics[width=0.48\textwidth]{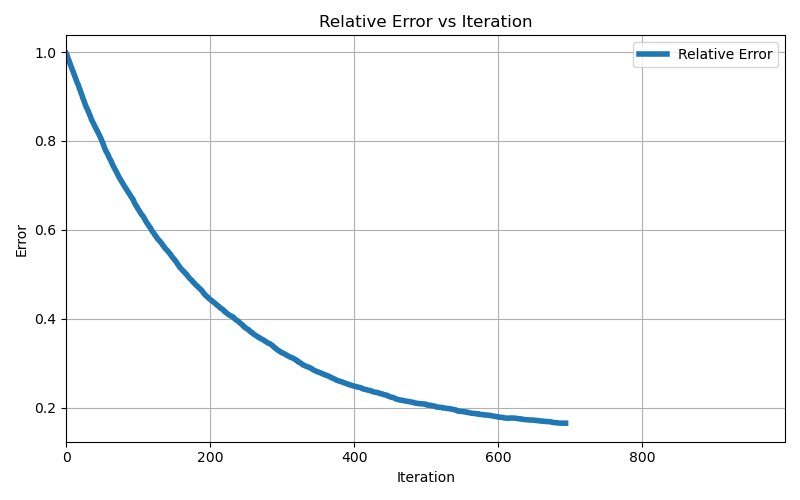}}
    \subfigure[]{\includegraphics[width=0.48\textwidth]{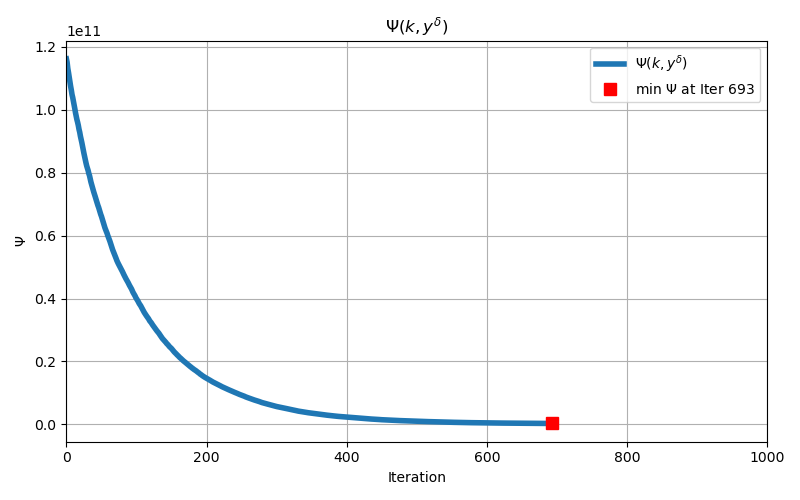}} 
        \caption{Results for $\delta_{\text{rel}} = 10^{-2}$ with $a = 1000$: (i) relative error in the reconstructed solution; (ii) progression of $\Psi(k, y^\delta)$ over the iterations.}
    \label{RE= 10000}
\end{figure}
\begin{figure}
    \centering
    \subfigure[]{\includegraphics[width=0.48\textwidth]{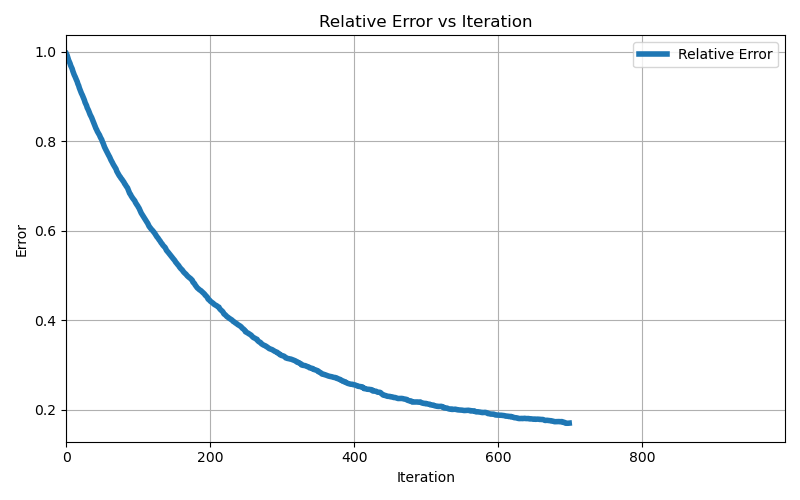}}
    \subfigure[]{\includegraphics[width=0.48\textwidth]{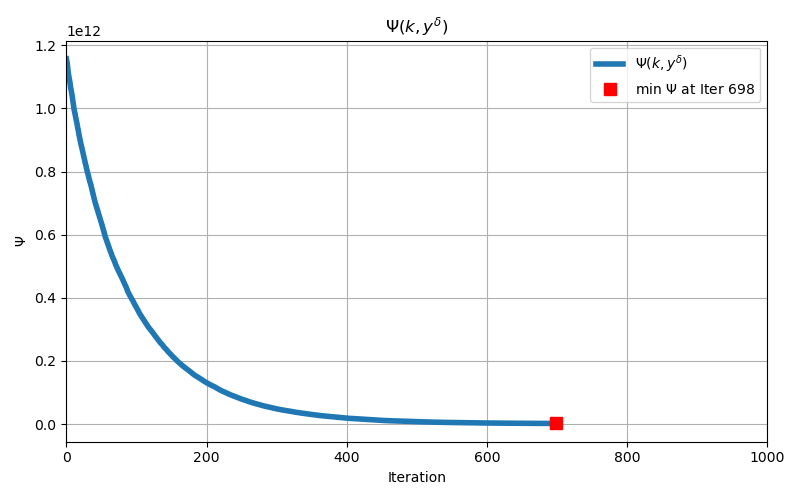}}  
       \caption{Results for $\delta_{\text{rel}} = 10^{-2}$ with $a = 10000$: (i) relative error in the reconstructed solution; (ii) progression of $\Psi(k, y^\delta)$ over the iterations.}
    \label{RE= 100000}
\end{figure}

\begin{table}[ht]
\centering
\begin{tabular}{lccccc}
\toprule
\multicolumn{6}{c}{$\delta_{\text{real}} = 0.01$} \\
\toprule
 \textbf{$a$} in Rule~\ref{Hanke Rule} &  \textbf{$k_*$} & \textbf{$E_{k_*}^{\delta}$} & \textbf{$\Psi(k_*, y^\delta)$} & PSNR & SSIM \\
\midrule
$100$ & $719$ & $1.62e^{-1}$ & $1.29e^{+8}$ & $28.11$ & $0.8251$ \\

$1000$ & $693$ & $1.65e^{-1}$ & $2.97e^{+8}$ & $27.95$ & $0.8215$  \\

 $10000$ &  $698$ & $1.70e^{-1}$ &$2.04e^{+9}$& $27.70$ & $0.8051$ \\
\bottomrule
\end{tabular}
\caption{Performance comparison for different values of the parameter $a$ in Rule~\ref{Hanke Rule} for $\delta_{\text{real}} = 0.01$, showing the stopping index $k_*$, reconstruction error $E_{k_*}^\delta$, $\Psi(k_*, y^\delta)$, and image quality metrics (PSNR and SSIM).}

\label{nonlinear-table}
\end{table}

\section{Conclusion}\label{Conc}
In this work, we have introduced an Iteratively Regularized Stochastic Gradient Descent (IRSGD) method for solving systems of nonlinear ill-posed inverse problems, incorporating the Hanke–Raus heuristic rule as a fully data-driven stopping criterion. A key advantage of this approach is that it operates without any prior knowledge of the noise level, thereby enhancing its applicability in practical scenarios. Our convergence analysis relies on the tangential cone condition, noise assumption and assumes Lipschitz continuity of the underlying operators. Moreover, we have established convergence results in the mean-squared norm and validated the method through comprehensive numerical experiments on both linear and nonlinear ill-posed problems.
The results demonstrate that the proposed IRSGD method, augmented with a damping factor, offers improved efficiency and faster convergence compared to standard stochastic gradient approaches.

Several promising directions for future research emerge from this study. First, it would be valuable to relax Assumption~\ref{new assumption} by introducing milder conditions, following the ideas in \cite{HumberS, Jahn2023}. Second, in line with \cite{Harshitdatadriven, clason2019bouligand, clason2019bouligand2, mittalbajpaiIR2}, the framework could be extended to handle non-smooth problems, thereby broadening its applicability. Finally, the development of adaptive step-size strategies could provide more practical guidelines for real-world implementation. Additionally, the heuristic stopping rule (based on a modified discrepancy principle) could be further refined to enable a deeper stochastic analysis of regularization methods such as those in \cite{clason2019bouligand2, kaltenbacher2008iterative, mittalbajpaiIR, mittalbajpaiIR2, Gaurav6}.

 \section*{Acknowledgments} Harshit Bajpai acknowledges the Anusandhan National Research Foundation (ANRF), India, for supporting his Ph.D. fellowship under grant CRG/2022/005491. Ankik Kumar Giri gratefully acknowledges ANRF for funding his research through the project CRG/2022/005491.

 \subsection*{Data Availability} Not applicable for this article.

\end{document}